\documentclass[amstex,12pt, amssymb]{article}

\usepackage{mathtext}
\usepackage[cp1251]{inputenc}
\usepackage[T2A]{fontenc}
\usepackage[dvips]{graphicx}
\usepackage{amsmath}
\usepackage{amssymb}
\usepackage{amsxtra}
\usepackage{latexsym}
\usepackage{ifthen}

\newcounter{lemma}[section]

\newcounter{corollary}[section]

\newcounter{remark}[section]

\newcounter{theorem}[section]

\newcounter{proposition}[section]

\newcounter{example}

\numberwithin{equation}{section}

\begin{document}

\markboth{N.~ILKEVYCH, D.~ROMASH, E.~SEVOST'YANOV}{\centerline{ON
THE GLOBAL BEHAVIOR OF MAPPINGS...}}

\def\cc{\setcounter{equation}{0}
\setcounter{figure}{0}\setcounter{table}{0}}

\overfullrule=0pt


\author{N.~ILKEVYCH, D.~ROMASH, E.~SEVOST'YANOV\footnote{Corresponding author}}

\title{\bf ON THE GLOBAL BEHAVIOR OF MAPPINGS AND THE CORRESPONDENCE
OF BOUNDARIES}

\date{\today}
\maketitle

\begin{abstract}
We consider families of mappings with moduli inequalities, having
different definition domains. Under some additional assumptions we
have proved that such families are uniformly equicontinuous. We have
considered four main cases: when mappings are homeomorphisms and
corresponding domains have simple geometry; when similar mappings
have branch points; when domains with complex geometry are
considered, but mappings still are homeomorphisms; and when similar
mappings have branch points. Sequences of domains are generally
assumed to converge to a kernel, and the characteristics of the
mappings must satisfy certain conditions on their growth. In some of
the four cases mentioned above, we also described properties of the
limit mapping. We also obtained the correspondence of the boundary
points of the kernel to the boundary points, and the inner points to
the inner points.
\end{abstract}


\section{Introduction}

This paper is devoted to the study of mappings satisfying upper
distortion estimates of the modulus of families of paths. We
consider the case when definition domains of mappings may be
different. The main subject of this study is distortion estimates
for the distance under mappings. For fixed domains, this problem has
been studied by many authors, including both homeomorphisms and
mappings with a branching. In particular, we could mentioned papers
on this matter, see, e.g., \cite{Cr$_1$}--\cite{Cr$_2$},
\cite{MRSY}, \cite{M}, \cite{RS$_1$} and \cite{SevSkv}. Considering
families of mappings of variable domains, we use some ideas
from~\cite{Suv}, where the theory of prime ends of variable domains
was studied. However, our research does not involve such a complex
theory, although, at the same time, we consider domains not only
``with prime ends'', but also with fairly smooth boundaries. We
should note that the manuscript is in line with the research and
development of the theory of mappings, modulus technique, and at the
same time fills the corresponding niche in the theory of mappings
with finite distortion, which have been actively studied recently
(see, e.g., \cite{MRSY}).

\medskip
We recall some definitions. A Borel function $\rho:{\Bbb
R}^n\,\rightarrow [0,\infty] $ is called {\it admissible} for the
family $\Gamma$ of paths $\gamma$ in ${\Bbb R}^n,$ if the relation
\begin{equation*}\label{eq1.4}
\int\limits_{\gamma}\rho (x)\, |dx|\geqslant 1
\end{equation*}
holds for all (locally rectifiable) paths $ \gamma \in \Gamma.$ In
this case, we write: $\rho \in {\rm adm} \,\Gamma .$ Let $p\geqslant
1,$ then {\it $p$-modulus} of $\Gamma $ is defined by the equality
\begin{equation*}\label{eq1.3gl0}
M_p(\Gamma)=\inf\limits_{\rho \in \,{\rm adm}\,\Gamma}
\int\limits_{{\Bbb R}^n} \rho^p (x)\,dm(x)\,.
\end{equation*}
We set $M(\Gamma):=M_n(\Gamma).$ Let $x_0\in {\Bbb R}^n,$
$0<r_1<r_2<\infty,$
\begin{equation}\label{eq1ED}
S(x_0,r) = \{ x\,\in\,{\Bbb R}^n : |x-x_0|=r\}\,, \quad B(x_0, r)=\{
x\,\in\,{\Bbb R}^n : |x-x_0|<r\}\end{equation}
and
\begin{equation}\label{eq1**}
A=A(x_0, r_1,r_2)=\left\{x\,\in\,{\Bbb R}^n:
r_1<|x-x_0|<r_2\right\}\,.\end{equation}
Given sets $E,$ $F\subset\overline{{\Bbb R}^n}$ and a domain
$D\subset {\Bbb R}^n$ we denote by $\Gamma(E,F,D)$ a family of all
paths $\gamma:[a,b]\rightarrow \overline{{\Bbb R}^n}$ such that
$\gamma(a)\in E,\gamma(b)\in\,F $ and $\gamma(t)\in D$ for $t \in
(a, b).$ Let $S_i=S(x_0, r_i),$ $i=1,2,$ where spheres $S(x_0, r_i)$
centered at $x_0$ of the radius $r_i$ are defined in~(\ref{eq1ED}).
Given $p\geqslant 1,$ a mapping $f:D\rightarrow \overline{{\Bbb
R}^n}$ is called a {\it ring $Q$-mapping at a point $x_0\in {\Bbb
R}^n$ with respect to $p$-modulus}, if the condition
\begin{equation} \label{eq2*!}
M_p(f(\Gamma(S_1, S_2, D)))\leqslant \int\limits_{A\cap D} Q(x)\cdot
\eta^p (|x-x_0|)\, dm(x)
\end{equation}
holds for all $0<r_1<r_2<\infty$ and any Lebesgue measurable
function $\eta:(r_1, r_2)\rightarrow [0, \infty]$ such that
\begin{equation}\label{eq8B_2}
\int\limits_{r_1}^{r_2}\eta(r)\,dr\geqslant 1\,.
\end{equation}
The inequalities of the form~(\ref{eq2*!}) were established for many
well-known classes of mappings, see e.g.
\cite[Definition~13.1]{Va}), \cite[Theorems~8.1, 8.5]{MRSY} and
\cite[Corollary~2]{Sal}.

\medskip
Recall that a mapping $f:D\rightarrow {\Bbb R}^n$ is called {\it
discrete} if the pre-image $\{f^{-1}\left(y\right)\}$ of each point
$y\,\in\,{\Bbb R}^n$ consists of isolated points, and {\it is open}
if the image of any open set $U\subset D$ is an open set in ${\Bbb
R}^n.$ A mapping $f:D\rightarrow {\Bbb R}^n$ is called {\it closed}
if $f(A)$ is closed in $f(D)$ whenever $A$ is closed in $D.$ Later,
in the extended space $\overline{{{\Bbb R}}^n}={{\Bbb
R}}^n\cup\{\infty\}$ we use the {\it spherical (chordal) metric}
$h(x,y)=|\pi(x)-\pi(y)|,$ where $\pi$ is a stereographic projection
$\overline{{{\Bbb R}}^n}$ onto the sphere
$S^n(\frac{1}{2}e_{n+1},\frac{1}{2})$ in ${{\Bbb R}}^{n+1},$ namely,
\begin{equation*}\label{eq3C}
h(x,\infty)=\frac{1}{\sqrt{1+{|x|}^2}}\,,\quad
h(x,y)=\frac{|x-y|}{\sqrt{1+{|x|}^2} \sqrt{1+{|y|}^2}}\,, \quad x\ne
\infty\ne y
\end{equation*}
(see \cite[Definition~12.1]{Va}). In what follows, ${\rm Int\,}A$
denotes the set of inner points of the set $A\subset \overline{{\Bbb
R}^n}.$ Recall that the set $U\subset\overline{{\Bbb R}^n}$ is
neighborhood of the point $z_0,$ if $z_0\in {\rm Int\,}A.$

\medskip
Given $E_1, E_2\subset\overline{{\Bbb R}^n}$ we set
$$h(E_1, E_2)=\inf\limits_{x\in E_1, y\in E_2}h(x, y)\,.$$
\medskip
Let $I$ be a fixed set of indices and let $D_i,$ $i\in I,$ be some
number of domains. Following~\cite[Sect.~2.4]{NP}, we say that a
family of domains $\{D_i\}_{i\in I}$ is {\it equi-uniform with
respect to $p$-modulus} if for any $r> 0$ there exists a number
$\delta> 0$ such that the inequality $M_p(\Gamma(F^{\,*},F,
D_i))\geqslant \delta$ holds for any $i\in I$ and any continua $F,
F^*\subset D_i$ such that $h(F)\geqslant r$ and $h(F^{\,*})\geqslant
r.$

\medskip
\begin{remark}\label{rem1}
Due to~\cite[Section~3]{MRSY} we say that a boundary $D$ is called
{\it strongly accessible with respect to $p$-modulus at $x_0\in
\partial D,$} if for any
neighborhood $U$ of the point $x_0\in\partial D$ there is a
neighborhood $V\subset U$ of this point, a compactum $F\subset D$
and a number $\delta>0$ such that $M_p(\Gamma(E, F, D))\geqslant
\delta$ for any continua $E\subset D$ such that $E\cap
\partial U\ne\varnothing\ne E\cap
\partial V.$ The boundary of a domain $D$ is called {\it strongly accessible with respect to
$p$-modulus,} if this is true for any $x_0\in
\partial D.$ When $p=n,$ prefix ``relative to $p$-modulus'' is omitted.
Uniform domains have strongly accessible boundaries, see e.g.
\cite[Remark~1]{SevSkv}. The inverse statement is also true at least
for $p=n,$ see \cite[Lemma~4.2]{Sev$_2$}.
\end{remark}

\medskip
\begin{remark}
Observe that, if $n=2,$ then every domain $D$ in ${\Bbb R}^2$ which
has a finite number of boundary components is uniform if and only if
$D$ is finitely connected on the boundary, see Corollary~6.8 in
\cite{Na$_2$}.
\end{remark}

Set
$$q_{x_0}(r):=\frac{1}{\omega_{n-1}r^{n-1}}\int\limits_{|x-x_0|=r}Q(x)\,dS\,,$$
where $dS$ is the element of the surface area of $S,$ and
$$q^{\,\prime}_{x_0}(r):=\frac{1}{\omega_{n-1}r^{n-1}}
\int\limits_{|x-x_0|=r}Q^{\,\prime}(x)\,dS\,,$$
$Q^{\,\prime}(x)=\max\{Q(x), 1\}.$ A domain $D\subset {\Bbb R}^n$ is
called {\it locally connected at the point} $x_0\in
\partial D,$ if for any neighborhood $U$ of point $x_0$ there
is a neighborhood $V\subset U$ of the same point such that $V\cap D$
is connected. A domain $D$ is called {\it locally connected on
$\partial D,$} if this domain is such at each point of its boundary.

\medskip Let $D_m,$ $m=1,2,\ldots ,$ be a sequence of domains in
${\Bbb R}^n,$ containing a fixed point $A_0.$ If there exists a ball
$B(A_0, \rho),$ $\rho>0,$ belonging to all $D_m,$ then the {\it
kernel} of the sequence $D_m,$ $m=1,2,\ldots ,$ with respect to
$A_0$ is the largest domain $D_0$ containing $x_0$ and such that for
each compact set $E$ belonging to $D_0$ there is $N>0$ such that $E$
belongs to $D_m$ for all $m\geqslant N.$ A largest domain is one
which contains any other domain having the same property. A sequence
of domains $D_m,$ $m=1,2,\ldots ,$ converges to a kernel $D_0$ if
any subsequence of $D_m$ has $D_0$ as its kernel.

\medskip
Let $D_m,$ $m=1,2,\ldots ,$ be a sequence of domains which converges
to a kernel $D_0.$ Then $D_m$ will be called {\it regular} with
respect to $D_0,$ if $D_m\subset D_0$ for all $m\in {\Bbb N}$ and,
for every $x_0\in \overline{D_0}$ and for every neighborhood $U$ of
$x_0$ there is a neighborhood $V$ of $x_0,$ $V\subset U,$ and
$M=M(U, x_0)\in {\Bbb N}$ such that $V\cap D_m$ is a non-empty
connected set for every $m\geqslant M(U, x_0).$

\medskip
Given $p\geqslant 1,$ a number $\delta>0,$ a fixed domain
$D_0\subset {\Bbb R}^n,$ $n\geqslant 2,$ a sequence of domains
$\frak{D}=\{D_m\}_{m=1}^{\infty}$ the kernel of which is $D_0,$ a
non-degenerate continuum $A\subset D_0,$
$A\subset\bigcap\limits_{m=1}^{\infty}D_m,$ and a Lebesgue
measurable function $Q:{\Bbb R}^n\rightarrow[0, \infty]$, we denote
by $\frak{F}_{Q, A, p, \delta}(D_0, \frak{D})$ some a family of ring
$Q$-homeomorphisms $f_m:D_m\rightarrow \overline{{\Bbb R}^n}$
satisfying~(\ref{eq2*!})--(\ref{eq8B_2}) at any point
$x_0\in\overline{D_0}$ for all $0<r_1<r_2<\infty$ such that
$h(f_m(A))\geqslant\delta$ and $h(\overline{{\Bbb R}^n}\setminus
f_m(D_m))\geqslant \delta.$ The concept of finite mean oscillation
(FMO) used below can be found, for example, in \cite{MRSY} or
\cite{Sev$_1$}. The following statement is true.

\medskip
\begin{theorem}\label{th2}
{\it\, Let $p\in (n-1, n]$ and let $f_m\in \frak{F}_{Q, A, p,
\delta}(D_0, \frak{D}),$ $m=1,2,\ldots ,$ be a sequence such that:

\medskip
1) the sequence of domains $D_m$ is regular with respect to~$D_0;$

\medskip
2) for every $m\in {\Bbb N},$ a domain $D_m$ is locally connected on
its boundary;

\medskip
3) the family $f_m(D_m)$ is equi-uniform with respect to $p$-modulus
over all $m\in {\Bbb N};$

\medskip
4) at least one of two following conditions hold: $Q$ has a finite
mean oscillation in $\overline{D_0},$ or
\begin{equation}\label{eq2_4}
\int\limits_{0}^{\beta(x_0)}\frac{dt}{t^{\frac{n-1}{p-1}}
q_{x_0}^{\,\prime\,\frac{1}{p-1}}(t)}=\infty
\end{equation}
for every $x_0\in \overline{D_0}$ and some $\beta(x_0)>0.$

\medskip
\textbf{I.} Then the family $f_m,$ $m=1,2,\ldots,$ is uniformly
equicontinuous in $\frak{D},$ i.e., for any $\varepsilon>0$ there is
$\delta=\delta(\varepsilon)>0$ such that $h(f_m(x),
f_m(y))<\varepsilon$ whenever $x,y\in D_m,$ $|x-y|<\delta$ and $m\in
{\Bbb N}.$

\medskip
\textbf{II.} Moreover, there is a subsequence $f_{m_k},$
$k=1,2,\dots ,$ which converges to $f$ locally uniformly in $D_0.$
In this case, $f$ has a continuous boundary extension
$f:\overline{D_0}\rightarrow \overline{{\Bbb R}^n}$ and $f$ is a
homeomorphism in $D_0.$ If $A_0$ is some compactum in $D_0,$ then
there is $\delta_2>0$ and $M_0\in {\Bbb N}$ such that $A_0\subset
D_m$ for all $m\geqslant M_0$ and $h(f_m(A_0),
\partial f_m(D_m))\geqslant \delta_2$ for every $m\geqslant M_0.$

\medskip
\textbf{III.} If $x_m\in D_{m},$ $m=1,2,\ldots,$ $f_m$ converges to
$f$ locally uniformly in $D_0$ and $x_m\rightarrow x_0$ as
$m\rightarrow \infty,$ then $f_{m}(x_m)\rightarrow f(x_0).$ If
$x_0\in
\partial D_0,$ then $f(x_0)\in \partial f(D_0).$}
\end{theorem}

\medskip
We also prove a result similar to Theorem~\ref{th2} for
non-injective mappings.

\medskip
Given $p\geqslant 1,$ a domain $D_0\subset {\Bbb R}^n,$ a set
$E\subset\overline{{\Bbb R}^n}$ and a number $\delta>0$, a fixed
domain $D_0\subset {\Bbb R}^n,$ $n\geqslant 2,$ a sequence of
domains $\frak{D}=\{D_m\}_{m=1}^{\infty}$ the kernel of which is
$D_0$ and a Lebesgue measurable function $Q:{\Bbb R}^n\rightarrow
[0, \infty]$ we denote by $\frak{R}_{Q, \delta, p, E}(D_0,
\frak{D})$ the family of all open discrete closed mappings
$f:D_m\rightarrow \overline{{\Bbb R}^n}\setminus E$
satisfying~(\ref{eq2*!})--(\ref{eq8B_2}) for any $0<r_1<r_2<\infty$
at any point $x_0\in\overline{D_0}$ such that the following
condition holds: for any domain $D^{\,\prime}_m:=f_m(D_m)$ there
exists a continuum $K_m\subset D^{\,\prime}_m$ such that
$h(K_m)\geqslant \delta$ and $h(f_m^{\,-1}(K_m), \partial
D_m)\geqslant \delta>0.$ The following statement holds.

\medskip
\begin{theorem}\label{th4_4}
{\,\it Let $p\in (n-1, n]$ and let $f_m\in \frak{R}_{Q, \delta, p,
E}(D_0, \frak{D}),$ $m=1,2,\ldots ,$ be a sequence such that the
conditions~1)--4) from Theorem~\ref{th2} hold. Let $E$ be a set of a
positive capacity for $p=n,$ and $E$ is arbitrary closed set for
$n-1<p<n.$

\medskip
\textbf{I.} Then the family $f_m,$ $m=1,2,\ldots,$ is uniformly
equicontinuous in $\frak{D},$ i.e., for any $\varepsilon>0$ there is
$\delta=\delta(\varepsilon)>0$ such that $h(f_m(x),
f_m(y))<\varepsilon$ whenever $x,y\in D_m,$ $|x-y|<\delta$ and $m\in
{\Bbb N}.$ Moreover, there is a subsequence $f_{m_k},$ $k=1,2,\dots
,$ which converges to $f$ locally uniformly in $D_0.$ In this case,
$f$ has a continuous boundary extension $f:\overline{D_0}\rightarrow
\overline{{\Bbb R}^n}.$ If $x_m\in D_{m},$ $m=1,2,\ldots,$ $f_m$
converges to $f$ locally uniformly in $D_0$ and $x_m\rightarrow x_0$
as $m\rightarrow \infty,$ then $f_{m}(x_m)\rightarrow f(x_0).$

\medskip
\textbf{II.} In addition, assume that

\medskip
5) $p=n$ and there is $r>0,$ does not depending on $m\in {\Bbb N},$
such that $h(E)\geqslant r$ whenever $E$ is any component of
$\partial f_m(D_m).$  If $A_0$ is some compactum in $D_0,$ then
there is $\delta_2>0$ and $M_0\in {\Bbb N}$ such that $A_0\subset
D_m$ and $h(f_m(A_0),
\partial f_m(D_m))\geqslant \delta_2$ for every $m\geqslant M_0.$
The mapping $f$ is boundary preserving: if $x_0\in \partial D_0,$
then $f(x_0)\in \partial f(D_0).$ If $\overline{B(x_0,
\varepsilon_0)}\subset D_0\cap\bigcap\limits_{m=1}^{\infty}D_m,$
then there is $\varepsilon_1>0$ does not depending on $m\in {\Bbb
N}$ such that
\begin{equation}\label{eq3A} B_h(f_m(x_0), \varepsilon_1)\subset f_m(D_m)\,,
\quad m=1,2,\ldots\,,
\end{equation}
where $B_h(y_0, r_0)=\{y\in\overline{{\Bbb R}^n}: h(y, y_0)<r_0\}.$
}
\end{theorem}

\medskip
Results similar to Theorems~\ref{th2} and~\ref{th4_4} also hold for
domains with prime ends. We now formulate these results.

\medskip
Definitions related to prime ends of domains can be found
in~\cite{Sev$_1$} and are therefore omitted. Consider the following
definition, which goes back to N\"akki~\cite{Na$_1$}, cf.~\cite{KR}.
The boundary of a domain $D$ in ${\Bbb R}^n$ is said to be {\it
locally quasiconformal} if every $x_0\in\partial D$ has a
neighborhood $U$ that admits a quasiconformal mapping $\varphi$ onto
the unit ball ${\Bbb B}^n\subset{\Bbb R}^n$ such that
$\varphi(\partial D\cap U)$ is the intersection of ${\Bbb B}^n$ and
a coordinate hyperplane. The sequence of cuts $\sigma_m,$
$m=1,2,\ldots ,$ is called {\it quasiconformally regular,} if
$\overline{\sigma_m}\cap\overline{\sigma_{m+1}}=\varnothing$ for
$m\in {\Bbb N}$ and, in addition, $d(\sigma_{m})\rightarrow 0$ as
$m\rightarrow\infty.$ If the end $K$ contains at least one
quasiconformally regular chain, then $K$ will be called {\it
quasiconformally regular}. We say that a bounded domain $D$ in
${\Bbb R}^n$ is {\it quasiconformally regular}, if $D$ can be
quasiconformally mapped to a domain with a locally quasiconformal
boundary whose closure is a compact in ${\Bbb R}^n,$ and, besides
that, every prime end in $D$ is quasiconformally regular. Note that
space $\overline{D}_P=D\cup E_D$ is metric, which can be
demonstrated as follows. If $g:D_0\rightarrow D$ is a quasiconformal
mapping of a domain $D_0$ with a locally quasiconformal boundary
onto some domain $D,$ then for $x, y\in \overline{D}_P$ we put:
\begin{equation}\label{eq5M}
\rho(x, y):=|g^{\,-1}(x)-g^{\,-1}(y)|\,,
\end{equation}
where the element $g^{\,-1}(x),$ $x\in E_D,$ is to be understood as
some (single) boundary point of $D_0.$ The specified boundary point
is unique and well-defined by~\cite[Theorem~2.1, Remark~2.1]{IS},
cf.~\cite[Theorem~4.1]{Na$_1$}. It is easy to verify that~$\rho$
in~(\ref{eq5M}) is a metric on $\overline{D}_P.$ If $g_*$ is another
quasiconformal mapping of a domain $D_*$ with locally quasiconformal
boundary onto $D$, then the corresponding metric
$\rho_*(p_1,p_2)=|{\widetilde{g_*}}^{-1}(p_1)-{\widetilde{g_*}}^{-1}(p_2)|$
generates the same convergence and, consequently, the same topology
in $\overline {D}_P$ as $\rho_0$ because $g_0\circ g_*^{-1}$ is a
quasiconformal mapping of $D_*$ and $D_0$, which extends, by Theorem
4.1 in~\cite{Na$_1$}, to a homeomorphism between $\overline {D_*}$
and $\overline {D_0}$. In the sequel, this topology in $\overline
{D}_P$ will be called the {\it topology of prime ends}; the
continuity of mappings $F\colon \overline
{D}_P\rightarrow\overline{D^{\,\prime}}_P$ will be understood
relative to this topology.

\medskip
Let $D_m,$ $m=1,2,\ldots ,$ be a sequence of domains which converges
to a kernel $D_0.$ Then $D_m$ will be called {\it regular} with
respect to $D_0$ in the terms of prime ends, if $D_m\subset D_0$ for
all $m\in {\Bbb N}$ and, for every $P_0\in E_{D_0}$ there is a
sequence $\sigma_k,$ $k=1,2,\ldots,$ with the following condition:
if $d_k$ is a domain in $P_0$ then there is $M=M(k)$ such that
$d_k\cap D_m$ is a non-empty connected set for every $m\geqslant
M(k).$ The following result holds.

\medskip
\begin{theorem}\label{th4}
{\it\, Let $p\in (n-1, n]$ and let $f_m\in \frak{F}_{Q, A, p,
\delta}(D_0, \frak{D}),$ $m=1,2,\ldots ,$ be a sequence such that:

\medskip
1) the sequence of domains $D_m$ is regular with respect to~$D_0$ in
terms of prime ends;

\medskip
2) for every $m\in {\Bbb N},$ a domain $D_m$ is quasiconformally
regular; $D_0$ is also quasiconformally regular;

\medskip
3) the family $f_m(D_m)$ is equi-uniform with respect to $p$-modulus
over all $m\in {\Bbb N}$;

\medskip
4) at least one of two following conditions hold: $Q$ has a finite
mean oscillation in $\overline{D_0},$ or (\ref{eq2_4}) holds
for every $x_0\in \overline{D_0}$ and some $\beta(x_0)>0.$

\medskip
\textbf{I.} Then the family $f_m,$ $m=1,2,\ldots,$ is uniformly
equicontinuous in $\frak{D}$ in terms of prime ends, i.e., for any
$\varepsilon>0$ there is $\delta=\delta(\varepsilon)>0$ such that
$h(f_m(x), f_m(y))<\varepsilon$ whenever $x,y\in D_m,$ $\rho(x,
y)<\delta,$ $m\in {\Bbb N}$ and $\rho$ is one of the metric in
$\overline{D_0}_P$ defined in~(\ref{eq5M}).

\medskip
\textbf{II.} Moreover, there is a subsequence $f_{m_k},$
$k=1,2,\dots ,$ which converges to $f$ locally uniformly in $D_0$
with respect $\rho.$ In this case, $f$ has a continuous boundary
extension $f:\overline{D_0}_P\rightarrow \overline{{\Bbb R}^n}$ and
$f$ is a homeomorphism in $D_0.$ If $A_0$ is some compactum in
$D_0,$ then there is $\delta_2>0$ and $M_0\in {\Bbb N}$ such that
$A_0\subset D_m$ for all $m\geqslant M_0$ and $h(f_m(A_0),
\partial f_m(D_m))\geqslant \delta_2$ for every $m\geqslant M_0.$

\medskip
\textbf{III.} If $x_m\in D_{m},$ $m=1,2,\ldots,$ $f_m$ converges to
$f$ locally uniformly in $D_0$ and $x_m\rightarrow P_0$ as
$m\rightarrow \infty,$ then $f_{m}(x_m)\rightarrow f(P_0).$ If
$P_0\in E_{D_0}:=\overline{D_0}_P\setminus D_0,$ then $f(P_0)\in
\partial f(D_0).$}
\end{theorem}

\medskip
\begin{theorem}\label{th5}
{\,\it Let $p\in (n-1, n]$ and let $f_m\in \frak{R}_{Q, \delta, p,
E}(D_0, \frak{D}),$ $m=1,2,\ldots ,$ be a sequence such that the
conditions~1)--4) from Theorem~\ref{th4} hold. Let $E$ be a set of a
positive capacity for $p=n,$ and $E$ is arbitrary closed set for
$n-1<p<n.$

\medskip
\textbf{I.} Then the family $f_m,$ $m=1,2,\ldots,$ is uniformly
equicontinuous in $\frak{D}$ with respect to prime ends, i.e., for
any $\varepsilon>0$ there is $\delta=\delta(\varepsilon)>0$ such
that $h(f_m(x), f_m(y))<\varepsilon$ whenever $x,y\in D_m,$ $\rho(x,
y)<\delta,$ $m\in {\Bbb N}$ and $\rho$ is one of the metric in
$\overline{D_0}_P.$ Moreover, there is a subsequence $f_{m_k},$
$k=1,2,\dots ,$ which converges to $f$ locally uniformly in $D_0.$
In this case, $f$ has a continuous boundary extension
$f:\overline{D_0}_P\rightarrow \overline{{\Bbb R}^n}.$ If $x_m\in
D_{m},$ $m=1,2,\ldots,$ $f_m$ converges to $f$ locally uniformly in
$D_0$ and $x_m\rightarrow P_0$ as $m\rightarrow \infty,$ then
$f_{m}(x_m)\rightarrow f(P_0).$

\medskip
\textbf{II.} Besides that, if the condition~5) of
Theorem~\ref{th4_4} holds, then for any compactum $A_0\subset D_0,$
$A_0\subset\bigcap\limits_{m=1}^{\infty}D_m,$ there exists
$\delta_1>0$ such that $h(f_m(A_0),
\partial f_m(D_m))\geqslant \delta_1$ for any $m=1,2,\ldots .$ The mapping
$f$ is boundary preserving: if $P_0\in E_{D_0},$ then $f(x_0)\in
\partial f(D_0).$  If $\overline{B(x_0, \varepsilon_0)}\subset
D_0\cap\bigcap\limits_{m=1}^{\infty}D_m,$ then there is
$\varepsilon_1>0$ such that the relation~(\ref{eq3A}) holds, where
$B_h(y_0, r_0)=\{y\in\overline{{\Bbb R}^n}: h(y, y_0)<r_0\}.$ }
\end{theorem}

\section{Preliminaries}

The following statement is true, see~\cite[Lemma~2.1]{Sev$_1$}.

\medskip
\begin{lemma}\label{lem1}{\it\, Let $p\geqslant 1,$ and let
$f:D\rightarrow \overline{{\Bbb R}^n}$ be an open discrete ring
$Q$-mapping at the point $b\in
\partial D$ with respect to $p$-modulus, $b\ne \infty,$
$f(D)=D^{\,\prime},$ a domain $D$ is locally connected at the point
$b,$ $C(f,
\partial D)\subset \partial D^{\,\prime}$
and $D^{\,\prime}$ is strongly accessible with respect to the
$p$-modulus at least at one point $y\in C(f, b),$ where $C(f, b)$ is
a cluster set of $f$ at $b.$ Suppose that there is $\varepsilon_0>0$
and some positive measurable function $\psi:(0,
\varepsilon_0)\rightarrow (0,\infty)$ such that
\begin{equation*}\label{eq7***}
0<I(\varepsilon,
\varepsilon_0)=\int\limits_{\varepsilon}^{\varepsilon_0}\psi(t)\,dt
< \infty
\end{equation*}
for any $\varepsilon\in(0, \varepsilon_0)$ and, in addition,
\begin{equation*}\label{eq5***}
\int\limits_{A(b, \varepsilon, \varepsilon_0)}
Q(x)\cdot\psi^{\,p}(|x-b|)
 \ dm(x) =o(I^p(\varepsilon, \varepsilon_0))\,,
\end{equation*}
where $A:=A(b, \varepsilon, \varepsilon_0)$ is defined
in~(\ref{eq1**}). Then $C(f, b)=\{y\}.$}
\end{lemma}

\medskip
Observe that, $C(f,
\partial D)\subset \partial D^{\,\prime}$ for homeomorphisms of $D$ onto
$D^{\,\prime},$ see e.g. \cite[Proposition~13.5]{MRSY}.

\medskip
We will prove the main assertions of the paper in a somewhat
more general form. In this regard, consider the following
definition. We will say that a point $x_0$ is a limit point of a
family of domains $\frak{D}$ in ${\Bbb R}^n$ if there exists a
sequence $D_m,$ $m=1,2,\ldots, $ in $\frak{D}$ and a sequence
$x_m\in D_m,$ such that $x_m\rightarrow x_0$ as $m\rightarrow
\infty.$ Let $\overline{\frak{D}}$ be a set of all such points
$x_0.$

\medskip
Given $p\geqslant 1,$ a number $\delta>0,$ a family $\frak{D}$ of
domains in ${\Bbb R}^n,$ $n\geqslant 2,$ and a Lebesgue measurable
function $Q:{\Bbb R}^n\rightarrow[0, \infty]$, we denote by
$\frak{H}^{p}_{Q, \delta}(\frak{D})$ a family of all homeomorphisms
$f:D\rightarrow \overline{{\Bbb R}^n}$ defined in some
$D\in\frak{D}$ and satisfying~(\ref{eq2*!})--(\ref{eq8B_2}) at any
point $x_0\in\overline{\frak{D}},$ $0<r_1<r_2<\infty,$ for which
there exists a continuum $A=A_f\subset D$ with $h(f(A))\geqslant
\delta.$

\medskip
The following lemma asserts the equicontinuity of a certain family
of mappings at a fixed point. Unlike the main result of
Theorem~\ref{th2}, this assertion does not assume the convergence of
the sequence of domains to its kernel, and the corresponding domains
in which the mappings are defined can be arbitrary. The following
statement is true.

\medskip
\begin{lemma}\label{lem2}
{\it\, Let $p\in (n-1, n],$ let $x_0\in \overline{\frak{D}}$ and let
$f_m\in \frak{H}^p_{Q, \delta}(\frak{D}),$ $m=1,2,\ldots ,$
$f_m:D_m\rightarrow \overline{{\Bbb R}^n}$ be a sequence such that:

\medskip
1) for every $m\in {\Bbb N},$ $D_m$ is locally connected at
$\partial D_m;$

2) for every neighborhood $U$ of $x_0$ there is a neighborhood $V$
of $x_0,$ $V\subset U,$ and $M=M(U, x_0)\in {\Bbb N}$ such that
$V\cap D_m$ is connected for every $m\geqslant M(U, x_0);$

3) there is $\varepsilon_0>0$ such that $A_m:=A_{f_m}\subset {\Bbb
R}^n\setminus B(x_0, \varepsilon_0),$ where $A_{f_m}$ is a continuum
corresponding to the definition of the class $\frak{H}^p_{Q,
\delta}(\frak{D})$ for $f_m;$

4) the family $f_m(D_m)$ is equi-uniform with respect to $p$-modulus
over all $m\in {\Bbb N};$

5) there exists $\varepsilon_0=\varepsilon_0(x_0)>0$ and a Lebesgue
measurable function $\psi:(0, \varepsilon_0)\rightarrow [0,\infty]$
such that
\begin{equation}\label{eq7***_2} I(\varepsilon,
\varepsilon_0):=\int\limits_{\varepsilon}^{\varepsilon_0}\psi(t)\,dt
< \infty\quad \forall\,\,\varepsilon\in (0, \varepsilon_0)\,,\quad
I(\varepsilon, \varepsilon_0)\rightarrow
\infty\quad\text{as}\quad\varepsilon\rightarrow 0\,,
\end{equation}
while
\begin{equation} \label{eq3.7.2_2}
\int\limits_{A(x_0, \varepsilon, \varepsilon_0)}
Q(x)\cdot\psi^{\,p}(|x-x_0|)\,dm(x) = o(I^p(\varepsilon,
\varepsilon_0))\,,\end{equation}
as $\varepsilon\rightarrow 0,$ where $A(x_0, \varepsilon,
\varepsilon_0)$ is defined~in (\ref{eq1**}). (If $x_0=\infty,$
(\ref{eq3.7.2_2}) means the similar relation for $x_0=0$ with
$\widetilde{Q}(x)=Q\left(\frac{x}{|x|^2}\right)$ instead of $Q(x)$).

\medskip
Then the family $f_m,$ $m=1,2,\ldots,$ is equicontinuous with
respect to $x_0,$ i.e., for any $\varepsilon>0$ there is
$\delta=\delta(x_0, \varepsilon)>0$ such that $h(f_m(x),
f_m(x^{\,\prime}))<\varepsilon$ whenever $x, x^{\,\prime}\in B(x_0,
\delta)\cap D_m$ and $m\in {\Bbb N}.$}
\end{lemma}

\medskip
\begin{proof}
We will mainly use the methodology that has been used for mappings
of a fixed domain, see, e.g., \cite{SevSkv}. Assume the contrary.
Then there exists $a>0$ such that for $\delta=1/k,$ $k=1,2,\ldots,$
there are $m_k\in {\Bbb N}$ and $x_k, x^{\,\prime}_k\in B(x_0,
1/k)\cap D_{m_k}$ such that $h(f_{m_k}(x_k),
f_{m_k}(x^{\,\prime}_k))\geqslant a.$ We may consider that the
sequence $m_k$ is increasing by $k.$ Otherwise, $m_k=k_0$ for
sufficiently large $k\in{\Bbb N}$ and some $k_0\in {\Bbb N},$
besides that, $x_0\in \overline{D_{k_0}}$ because $f_{k_0}$ is
defined at $x_k\cap D_{k_0}$ and $x_k\rightarrow x_0$ as
$k\rightarrow\infty.$ Now, from the assumption above we have that
$h(f_{k_0}(x_k), f_{k_0}(x^{\,\prime}_k))\geqslant a.$ The latter
contradicts with Lemma~\ref{lem1}. So, $m_k$ may be chosen as an
increasing sequence. Without loss of generality, going to
renumbering, if required, we may consider that $h(f_{m_k}(x_k),
f_{m_k}(x^{\,\prime}_k))\geqslant a$ holds for $m$ instead of $m_k,$
i.e.,
\begin{equation}\label{eq6***}
h(f_m(x_m), f_m(x^{\,\prime}_m))\geqslant a/2\qquad \forall\,\,m\in
{\Bbb N}\,.
\end{equation}
We may assume that $x_0\ne \infty.$ Now, we use the condition 2):
for every neighborhood $U$ of $x_0$ there is a neighborhood $V$ of
$x_0,$ $V\subset U,$ and $M=M(U, x_0)\in {\Bbb N}$ such that $V\cap
D_m$ is connected for every $m\geqslant M(U, x_0).$ Thus, given a
ball $U_1=B(x_0, 2^{\,-1})$ there is $V_1\subset U_1$ and $M(1)\in
{\Bbb N}$ such that $D_m\cap V_1$ is connected for $m\geqslant
M(1).$ Since $x_m, x_m^{\,\prime}\rightarrow x_0$ as $m\rightarrow
\infty,$ there is $m_1\geqslant M(1)$ such that $x_{m_1},
x^{\,\prime}_{m_1}\in V_1\cap D_{m_1}.$ Follow, given a ball
$U_2=B(x_0, 2^{\,-2})$ there is $V_2\subset U_2$ and $M(2)\in {\Bbb
N}$ such that $D_m\cap V_2$ is connected for $m\geqslant M(2).$
Since $x_m, x_m^{\,\prime}\rightarrow x_0$ as $m\rightarrow \infty,$
there is $m_2\geqslant \max\{M(2), m_1\}$ such that $x_{m_2},
x^{\,\prime}_{m_2}\in V_2\cap D_{m_2}.$ Etc. In general, given a
ball $U_k=B(x_0, 2^{\,-k})$ there is $V_k\subset U_k$ and $M(k)\in
{\Bbb N}$ such that $D_m\cap V_k$ is connected for $m\geqslant
M(k).$ Since $x_m, x_m^{\,\prime}\rightarrow x_0$ as $m\rightarrow
\infty,$ there is $m_k\geqslant \max\{M(k), m_{k-1}\}$ such that
$x_{m_k}, x^{\,\prime}_{m_k}\in V_k\cap D_{m_k}.$ Relabeling, if
required, we may consider that the same sequences $x_m$ and
$x^{\prime}_m$ satisfy the above conditions, i.e., given a ball
$U_m=B(x_0, 2^{\,-m}),$ $m=1,2,\ldots,$ there is $V_m\subset U_m$
such that $x_m, x^{\,\prime}_m\in D_m\cap V_m$ while $D_m\cap V_m$
is connected. We join the points $x_m$ and $x^{\,\prime}_m$ by a
path $\gamma_m:[0,1]\rightarrow D_m\cap V_m$ such that
$\gamma_m(0)=x_m,$ $\gamma_m(1)=x^{\,\prime}_m$ and $\gamma_m(t)\in
D_m\cap V_m$ for $t\in (0,1),$ see~Figure~\ref{fig6_1}.
\begin{figure}[h]
\centerline{\includegraphics[scale=0.5]{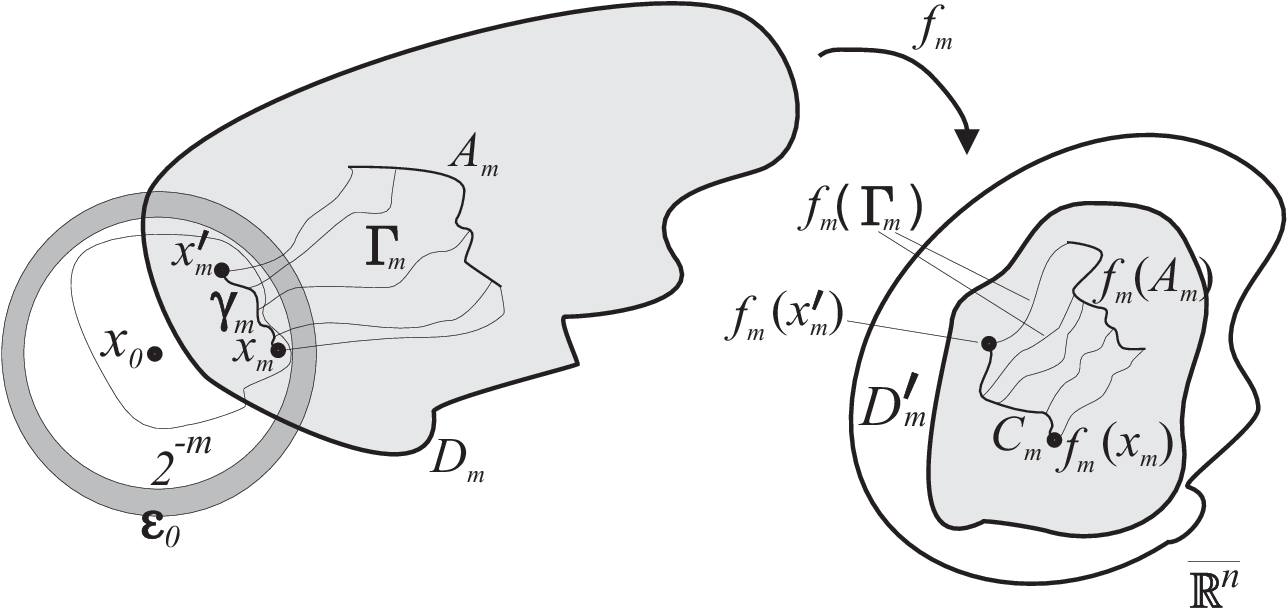}} \caption{To
prove Lemma~\ref{lem1}}\label{fig6_1}
\end{figure}
Denote by $C_m$ the image of the path $\gamma_m(t)$ under $f_m.$ It
follows from relation (\ref{eq6***}) that
\begin{equation}\label{eq5.1}
h(C_m)\geqslant a/2\qquad\forall\, m\in {\Bbb N}\,,
\end{equation}
where $h(C_m)$ denotes the chordal diameter of the set $C_m.$

\medskip
Let $\Gamma_m$ be a family of paths joining $|\gamma_m|$ and $A_m$
in $D_m,$ where $A_m$ is a continuum from the condition~3). By the
definition of the class $\frak{H}^{p}_{Q, \delta}(\frak{D}),$ the
relations~(\ref{eq2*!})--(\ref{eq8B_2}) hold at $x_0.$ Therefore,
\begin{equation}\label{eq10_2}
M_p(f_m(\Gamma_m))\leqslant \int\limits_{A(x_0, 2^{\,-m},
\varepsilon_0)} Q(x)\cdot \eta^p(|x-x_0|)\ dm(x)
\end{equation}
for every Lebesgue measurable function $\eta:(2^{\,-m},
\varepsilon_0)\rightarrow [0,\infty ]$ such that
$\int\limits_{2^{\,-m}}^{\varepsilon_0}\eta(r)\,dr \geqslant 1.$
Observe that, the function
$$\eta(t)=\left\{
\begin{array}{rr}
\psi(t)/I(2^{\,-m}, \varepsilon_0), & t\in (2^{\,-m},
\varepsilon_0),\\
0, & t\in {\Bbb R}\setminus (2^{\,-m}, \varepsilon_0)\,,
\end{array}
\right. $$ where $I(\varepsilon,
\varepsilon_0):=\int\limits_{\varepsilon}^{\varepsilon_0}\psi(t)\,dt,$
satisfies the condition~(\ref{eq8B_2}) with $r_1:=2^{\,-m},$
$r_2:=\varepsilon_0.$ Therefore, it follows from~(\ref{eq3.7.2_2})
and (\ref{eq10_2}) that
\begin{equation}\label{eq11_2}
M_p(f_m(\Gamma_m))\leqslant \alpha(2^{\,-m})\rightarrow 0
\end{equation}
as $m\rightarrow \infty,$ where $\alpha(\varepsilon)$ is some
nonnegative function tending to zero as $\varepsilon\rightarrow 0,$
which exists due to the condition~(\ref{eq3.7.2_2}).

\medskip
On the other hand, observe that $f_m(\Gamma_m)=\Gamma(C_m, f_m(A_m),
D_m^{\,\prime}).$ By the hypothesis of the lemma,
$h(f_m(A_m))\geqslant \delta$ for all $m\in {\Bbb N}.$ Therefore,
due to (\ref{eq5.1}), $h(f_m(A_m))\geqslant \delta_1$ and
$h(C_m)\geqslant\delta_1,$ where $\delta_1:=\min\{\delta, a/2\}.$
Using the fact that the domains $D_m^{\,\prime}:=f_m(D_m)$ are
equi-uniform with respect to $p$-modulus, we conclude that there
exists $\sigma>0$ such that
$$M_p(f_m(\Gamma_m))=M_p(\Gamma(C_m, f_m(A_m),
D_m^{\,\prime}))\geqslant \sigma\qquad\forall\,\, m\in {\Bbb N}\,,$$
which contradicts the condition~(\ref{eq11_2}). The obtained
contradiction indicates that the assumption of the absence of
equicontinuity of the family $\frak{H}^{p}_{Q, \delta}(\frak{D})$ at
$x_0$ was incorrect. The obtained contradiction completes the proof
of the lemma.~$\Box$
\end{proof}

\medskip
The following statement may be found in \cite[Lemmas~1.3,
1.4]{Sev$_1$}.

\medskip
\begin{proposition}\label{pr1}
{\it\, Let $Q:{\Bbb R}^n\rightarrow [0,\infty],$ $n\geqslant 2,$ be
a Lebesgue measurable function and let $x_0\in {\Bbb R}^n.$ Let
$0<p\leqslant n.$ Assume that either of the following conditions
holds

\noindent (a) $Q\in FMO(x_0),$

\noindent (b)
$q_{x_0}(r)\,=\,O\left(\left[\log{\frac1r}\right]^{n-1}\right)$ as
$r\rightarrow 0,$

\noindent (c) for some small $\delta_0=\delta_0(x_0)>0$ we have the
relations
\begin{equation*}\label{eq5***B}
\int\limits_{\delta}^{\delta_0}\frac{dt}{t^{\frac{n-1}{p-1}}
q_{x_0}^{\frac{1}{p-1}}(t)}<\infty,\qquad 0<\delta<\delta_0,\qquad
\int\limits_{0}^{\delta_0}\frac{dt}{t^{\frac{n-1}{p-1}}q_{x_0}^{\frac{1}{p-1}}(t)}=\infty\,.
\end{equation*}
Then  there exist a number $\varepsilon_0\in(0,1)$ and a function
$\psi\geqslant 0$ satisfying~(\ref{eq7***_2}) such that the relation
(\ref{eq3.7.2_2}) holds.}
\end{proposition}

\medskip
Uniting Lemma~\ref{lem2} and Proposition~\ref{pr1}, we obtain the
following.

\medskip
\begin{theorem}\label{th1}
{\it\, Let $p\in (n-1, n],$ let $x_0\in \overline{\frak{D}}$ and let
$f_m\in \frak{H}^p_{Q, \delta}(\frak{D}),$ $m=1,2,\ldots ,$
$f_m:D_m\rightarrow \overline{{\Bbb R}^n}$ be a sequence such that:

\medskip
1) for every $m\in {\Bbb N},$ $D_m$ is locally connected at
$\partial D_m;$

2) for every neighborhood $U$ of $x_0$ there is a neighborhood $V$
of $x_0,$ $V\subset U,$ and $M=M(U, x_0)\in {\Bbb N}$ such that
$V\cap D_m$ is connected for every $m\geqslant M=M(U, x_0);$

3) there is $\varepsilon_0>0$ such that $A_m:=A_{f_m}\subset {\Bbb
R}^n\setminus B(x_0, \varepsilon_0),$ where $A_{f_m}$ is a continuum
corresponding to the definition of the class $\frak{H}^p_{Q,
\delta}(\frak{D})$ for $f_m;$

4) the family $f_m(D_m)$ is equi-uniform with respect to $p$-modulus
over all $m\in {\Bbb N};$

5)  at least one of two following conditions hold: $Q$ has a finite
mean oscillation in $\overline{D_0},$ or~(\ref{eq2_4}) holds for
every $x_0\in \overline{D_0}$ and some $\beta(x_0)>0.$

\medskip
Then the family $f_m,$ $m=1,2,\ldots,$ is equicontinuous with
respect to $x_0,$ i.e., for any $\varepsilon>0$ there is
$\delta=\delta(x_0, \varepsilon)>0$ such that $h(f_m(x),
f_m(x^{\,\prime}))<\varepsilon$ whenever $x, x^{\,\prime}\in B(x_0,
\delta)\cap D_m$ and $m\in {\Bbb N}.$}
\end{theorem}

\section{On image of the continuum under open discrete and closed mappings}

The following statement holds, see, e.g.,
\cite[Theorem~1.I.5.46]{Ku}).

\medskip
\begin{proposition}\label{pr2}
{\it\, Let $A$ be a set in a topological space $X.$ If the set $C$
is connected and $C\cap A\ne \varnothing\ne C\setminus A,$ then
$C\cap
\partial A\ne\varnothing.$}
\end{proposition}

\medskip
Let $D$ be a domain in ${\Bbb R}^n,$ $n\geqslant 2,$ and let $f:D
\rightarrow {\Bbb R}^n$ (or $f:D\rightarrow \overline{{\Bbb R}^n})$
be a discrete open mapping, $\beta: [a,\,b)\rightarrow {\Bbb R}^n$
be a path, and $x\in\,f^{-1}(\beta(a)).$ A path $\alpha:
[a,\,b)\rightarrow D$ is called a {\it total $f$-lifting} of $\beta$
starting at $x,$ if $(1)\quad \alpha(a)=x\,;$ $(2)\quad
(f\circ\alpha)(t)=\beta(t)$ for any $t\in [a, b).$

\medskip
\begin{proposition}\label{pr4_a}
{\it\, Let $f:D \rightarrow {\Bbb R}^n$ be a discrete open and
closed (boundary preserving) mapping, $\beta: [a,\,b)\rightarrow
f(D)$ be a path, and $x\in\,f^{\,-1}\left(\beta(a)\right).$ Then
$\beta$ has a total $f$-lif\-ting $\alpha:[a, b)\rightarrow D$
starting at $x$ (see, e.g., \cite[Lemma~3.7]{Vu}). Moreover, if
$\beta(t)\rightarrow \partial f(D)$ as $t\rightarrow b,$ then
$\alpha(t)\rightarrow \partial D$ as $t\rightarrow b$ (see, e.g.,
Lemma~3.12 in \cite{MRV}).}
\end{proposition}

The following statement holds, see e.g. \cite[Theorem~7.2]{MRSY}.

\medskip
\begin{lemma}\label{lem5} {\bf(V\"{a}is\"{a}l\"{a}'s lemma
on the weak flatness of inner points).}
{\sl\, Let $n\geqslant 2 $, let $D$ be a domain in $\overline{{\Bbb
R}^n},$ and let $x_0\in D.$ Then for each $P>0$ and each
neighborhood $U$ of point $x_0$ there is a neighborhood $V\subset U$
of the same point such that $M(\Gamma(E, F, D))> P$ for any continua
$E, F \subset D $ intersecting $\partial U$ and $\partial V.$}
\end{lemma}

\medskip
The proof of Lemma~\ref{lem5} is essentially given by
V\"{a}is\"{a}l\"{a} in~\cite[(10.11)]{Va}, however, we have also
given a formal proof, see \cite[Lemma~2.2]{SevSkv}.~$\Box$

\medskip
\begin{lemma}\label{lem4}{\it\, Let
$D_0\subset {\Bbb R}^n,$ be a domain in ${\Bbb R}^n,$ $n\geqslant
2,$ and let $\frak{D}=\{D_m\}_{m=1}^{\infty}$ be a sequence of
domains the kernel of which is $D_0.$ Let $A$ be a non-degenerate
continuum in $D_0,$ $A\subset\bigcap\limits_{m=1}^{\infty}D_m,$ and
let $Q:{\Bbb R}^n\rightarrow[0, \infty]$ be a Lebesgue measurable
function. In addition, let $f_m:D_m\rightarrow {\Bbb R}^n$ be a
sequence of open discrete and closed mappings satisfying
relations~(\ref{eq2*!})--(\ref{eq8B_2}) at any point
$x_0\in\overline{D_0}$ for all $0<r_1<r_2<d(x_0,
\partial D_m)$ such that $h(f_m(A))\geqslant \delta>0$ for any
$m=1,2,\ldots$ and $p=n$ for some $\delta>0.$ Assume that, $Q$ has a
finite mean oscillation in $A,$ or the relation~(\ref{eq2_4}) holds
for every $x_0\in A$ and some $\beta(x_0)>0.$ If there is $r>0$ such
that $h(E)\geqslant r$ whenever $E$ is any component of $\partial
f_m(D_m),$ then there exists $\delta_1>0$ such that $h(f_m(A),
\partial f_m(D_m))\geqslant \delta_1$ for any $m=1,2,\ldots .$}
\end{lemma}

\medskip
\begin{proof}
Observe that, $h(A, \partial D_m)\geqslant r_*>0$ for some $r_*>0,$
and all $m\in {\Bbb N}$ except finite numbers. Otherwise, there are
sequences $x_m\in A$ and $y_m\in \partial D_m$ with $h(x_m,
y_m)\rightarrow 0$ as $m\rightarrow\infty.$ We may consider that
$x_m, y_m\rightarrow x_0$ as $m\rightarrow\infty.$ Obviously,
$x_0\in A$ by the closeness of $A.$ Now, there is $\varepsilon_0>0$
such that $\overline{B(x_0, \varepsilon_0)}\subset D_0$ and
$\overline{B(x_0, \varepsilon_0)}$ belongs to all $D_m$ except of
finite numbers. However, the latter contradicts the condition
$y_m\rightarrow x_0$ as $m\rightarrow\infty$ while $y_m\in
\partial D_m,$ because in this case $\overline{B(x_0,
\varepsilon_0)}\cap \partial D_m\ne\varnothing$ for infinitely many
$m\in {\Bbb N}.$

\medskip
We firstly consider the case when the relation~(\ref{eq2_4}) holds.
Since for any
$0<\varepsilon_1<\widetilde{\varepsilon_1}<\min\{\beta(x_0),
\frac{1}{2}r_*\}$ we have that
$$\int\limits_{\varepsilon_1}^{\widetilde{\varepsilon_1}}\frac{dt}{t
q_{x_0}^{\,\prime\,\frac{1}{n-1}}(t)}\leqslant
\int\limits_{\varepsilon_1}^{\widetilde{\varepsilon_1}}\frac{dt}{t}<\infty\,,
$$
the relation~(\ref{eq2_4}) implies that, given
$0<\widetilde{\varepsilon_1}<r_*\leqslant h(A, \partial D_m)$ we may
find $\varepsilon_1=\varepsilon_1(x_0)$ such that
$\int\limits_{\varepsilon_1}^{\widetilde{\varepsilon_1}}\frac{dt}{t
q_{x_0}^{\,\prime\,\frac{1}{n-1}}(t)}>1\,.$
We cover $A$ with the balls $B(x_0, \varepsilon_1),$ $x_0\in A,$
where $\varepsilon_1$ as above. By Heine-Borel-Lebesgue lemma we may
find a finite subcover $B(x_1, \varepsilon_1),$ $B(x_2,
\varepsilon_2),\ldots ,B(x_N, \varepsilon_N)$ such that
$ A\subset \bigcup\limits_{i=1}^nB(x_i, \varepsilon_i)\,,$ $A\subset
\bigcup\limits_{i=1}^nB(x_i, 2\varepsilon_i)\,.$
By the choice of $\varepsilon_i>0,$ $\varepsilon_i<\min\{\beta(x_0),
\frac{1}{2}r_*\},$ we have that $B(x_i, 2\varepsilon_i)\subset D_m.$
Since $D_m$ is a domain in ${\Bbb R}^n,$ $m=1,2,\ldots ,$ $\partial
D_m\ne\varnothing.$ In addition, since $C(\partial D_m, f_m)\subset
\partial f_m(D_m)$ whenever $f_m$ is open, discrete and closed (see
\cite[Theorem~3.3]{Vu}), $\partial f_m(D_m)\ne \varnothing,$ as
well. Thus, the quantity $h(f_m(A),
\partial f_m(D_m))$ is well-defined, so the formulation of the lemma is
correct.

\medskip
Let us prove Lemma~\ref{lem4} by the contradiction, partially using
the approach of the proof of Lemma~4.1 in \cite{SevSkv}. Suppose
that the conclusion of the lemma is not true. Then for each $k\in
{\Bbb N}$ there is some number $m=m_k$ such that $h(f_{m_k}(A),
\partial f_{m_k}(D_{m_k}))<1/k.$ In order not to complicate the notation,
we will further assume that $h(f_{m}(A),
\partial f_m(D_m))<1/m,$ $m=1,2,\ldots. $ Note that
the set $f_{m}(A)$ is compact as a continuous image of a compact set
$A\subset D_m$ under the mapping~$f_{m}.$ In this case, there are
elements $x_m\in f_{m}(A)$ and $y_m\in \partial f_m(D_m)$ such that
$h(f_{m}(A),
\partial f_m(D_m))=h(x_m, y_m)<1/m.$ Let $E_m$ be a component of $\partial
f_m(D_m)$ containing $y_m.$ Due to the compactness of
$\overline{{\Bbb R}^n},$ we may assume that $y_m\rightarrow y_0$ as
$m\rightarrow \infty;$ then also $x_m\rightarrow y_0$ as
$m\rightarrow \infty.$

\medskip
Put $P>0$ and $U=B_h(y_0, r_0)=\{y\in \overline{{\Bbb R}^n}: h(y,
y_0)<r_0\},$ where $2r_0:=\min\{r/2, \delta/2\},$ where $r$ and
$\delta$ are numbers from the condition of the lemma. Observe that
$E_m\cap U\ne\varnothing\ne E_m\setminus U$ for sufficiently large
$m\in{\Bbb N},$ since $y_m\rightarrow y_0$ as $m\rightarrow \infty,$
$y_m\in E_m;$ besides that, $h(E_m)\geqslant r>r/2\geqslant 2r_0$
and $h(U)\leqslant 2r_0.$ Observe that, $E_m$ is closed (see
\cite[Theorem~1.III.5.46]{Ku}). Now, $E_m$ is a continuum, so that
$E_m\cap
\partial U\ne\varnothing$ by Proposition~\ref{pr2}. Similarly,
$f_{m}(A)\cap U\ne\varnothing\ne f_{m}(A)\setminus U$ for
sufficiently large $m\in{\Bbb N},$ since $x_m\rightarrow y_0$ as
$m\rightarrow \infty,$ $x_m\in f_{m}(A);$ besides that
$h(f_{m}(A))\geqslant \delta>\delta/2\geqslant 2r_0$ and
$h(U)\leqslant 2r_0.$ Since $f_m(A)$ is a continuum, $f_m(A)\cap
\partial U\ne\varnothing$ by Proposition~\ref{pr2}.
By the proving above, $E_m\cap
\partial U\ne\varnothing\ne f_m(A)\cap
\partial U$
for sufficiently large $m\in {\Bbb N}.$ By Lemma~\ref{lem5} there is
$V\subset U,$ $V$ is a neighborhood of $y_0,$ such that
\begin{equation}\label{eq9}
M(\Gamma(E, F, \overline{{\Bbb R}^n}))>P
\end{equation}
for any continua $E, F\subset \overline{{\Bbb R}^n}$ with $E\cap
\partial U\ne\varnothing\ne E\cap \partial V$ and $F\cap \partial
U\ne\varnothing\ne F\cap \partial V.$
Arguing similarly to above, we may prove that
$
E_m\cap
\partial V\ne\varnothing\ne f_m(A)\cap
\partial V
$
for sufficiently large $m\in {\Bbb N}.$ Thus, by~(\ref{eq9})
\begin{equation}\label{eq9B}
M(\Gamma(f_m(A), E_m,\overline{{\Bbb R}^n}))>P
\end{equation}
for sufficiently large $m=1,2,\ldots .$ We now prove that the
relation~(\ref{eq9B}) contradicts the definition of $f_m$
in~(\ref{eq2*!})--(\ref{eq8B_2}).

\medskip
Indeed, let $\gamma:[0, 1]\rightarrow \overline{{\Bbb R}^n}$ be a
path in $\Gamma(f_m(A), E_m, \overline{{\Bbb R}^n}),$ i.e.,
$\gamma(0)\in f_m(A),$ $\gamma(1)\in E_m$ and $\gamma(t)\in
\overline{{\Bbb R}^n}$ for $t\in (0, 1).$ Let
$t_m=\sup\limits_{\gamma(t)\in f_m(D_m)}t$ and let
$\alpha_m(t)=\gamma|_{[0, t_m)}.$ Let $\Gamma_m$ consists of all
such paths $\alpha_m,$ now $\Gamma(f_m(A), E_m, \overline{{\Bbb
R}^n})>\Gamma_m$ and by the minorization principle of the modulus
(see~\cite[Theorem~1]{Fu})
\begin{equation}\label{eq12}
M(\Gamma_m)\geqslant M(\Gamma(f_m(A), E_m, \overline{{\Bbb
R}^n}))\,.
\end{equation}
Let $\Delta_m$ be a family of all total $f_m$-liftings of $\Gamma_m$
starting at $A$ (it exists by Proposition~\ref{pr4_a}). Now
\begin{equation}\label{eq13}
M(f_m(\Delta_m))=M(\Gamma_m)\,.
\end{equation}
Let $\beta_m:[0, t_m)\rightarrow D_m,$ $\beta_m\in \Delta_m,$
$f_m\circ\beta_m=\alpha_m.$ By Proposition~\ref{pr4_a}
$\beta_m(t)\rightarrow
\partial D_m$ as $t\rightarrow t_m.$

\medskip
We now show that
\begin{equation}\label{eq7C}
\Delta_{m}>\bigcup\limits_{i=1}^N\Gamma(S(x_i, \varepsilon_i),
S(x_i, 2\varepsilon_i), A(x_i, \varepsilon_i, 2\varepsilon_i))\,.
\end{equation}
Indeed, let $\gamma\in \Delta_{m},$ in other words, $\gamma:[0,
t_m)\rightarrow D_m,$ $\gamma(0)\in A$ and $\gamma(t)\rightarrow
\partial D_m$ as $t\rightarrow t_m-0.$ Now, there is $1\leqslant i\leqslant
N$ such that $|\gamma|\cap B(x_i, \varepsilon_i)\ne\varnothing\ne
|\gamma|\cap (D_m\setminus B(x_i, \varepsilon_i)).$ Therefore, by
Proposition~\ref{pr2} there is $0<t^{\,*}_1<t_m$ such that
$\gamma(t^{\,*}_1)\in S(x_i, \varepsilon_i).$ We may assume that
$\gamma(t)\not\in B(x_i, \varepsilon_i)$ for $t>t^{\,*}_1.$ Put
$\gamma_1:=\gamma|_{[t^{\,*}_1, t_m]}.$ Similarly, $|\gamma_1|\cap
B(x_i, 2\varepsilon_i)\ne\varnothing\ne |\gamma_1|\cap (D_m\setminus
B(x_i, 2\varepsilon_i)).$ By Proposition~\ref{pr2} there is
$t^{\,*}_1<t^{\,*}_2<t_m$ with $\gamma(t^{\,*}_2)\in S(x_i,
2\varepsilon_i).$ We may assume that $\gamma(t)\in B(x_i,
2\varepsilon_i)$ for $t<t^{\,*}_2.$ Put
$\gamma_2:=\gamma|_{[t^{\,*}_1, t^{\,*}_2]}.$ Then, the path
$\gamma_2$ is a subpath of $\gamma,$ which belongs to the family
$\Gamma(S(x_i, \varepsilon_i), S(x_i, 2\varepsilon_i), A(x_i,
\varepsilon_i, 2\varepsilon_i)).$ Thus, the relation~(\ref{eq7C}) is
established.

\medskip
On the other hand, we set
$$\psi_i(t)= \left \{\begin{array}{rr}
1/[tq_{x_i}^{\,\prime\frac{1}{n-1}}(t)]\ , & \ t\in (\varepsilon_i,
2\varepsilon_i)\ ,
\\ 0\ ,  &  \ t\notin (\varepsilon_i,
2\varepsilon_i)\ ,
\end{array} \right. $$
where $q^{\,\prime}_{x_i}(r):=\frac{1}{\omega_{n-1}r^{n-1}}
\int\limits_{|x-x_i|=r}Q^{\,\prime}(x)\,dS,$
$Q^{\,\prime}(x)=\max\{Q(x), 1\}.$
Now,
\begin{equation}\label{eq3}
\int\limits_{A(x_i, \varepsilon_i, 2\varepsilon_i)}
Q^{\,\prime}(x)\cdot\psi^n(|x-x_i|)\,dm(x)= \omega_{n-1} I_i\,,
\end{equation}
where $I_i=\int\limits_{\varepsilon_i}^{2\varepsilon_i}
\frac{dt}{tq_{x_i}^{^{\,\prime}\frac{1}{n-1}}(t)}.$
Note that, the function $\eta_i(t)= \psi(t)/I_i,$ $t\in
(\varepsilon_i, 2\varepsilon_i),$ satisfies~(\ref{eq8B_2}), because
$\int\limits_{\varepsilon_i}^{2\varepsilon_i}\eta_i(t)\,dt=1.$ Now,
it follows from~(\ref{eq3}) that
\begin{equation}\label{eq8C}
M(f_{m}(\Gamma(S(x_i, \varepsilon_i), S(x_i, 2\varepsilon_i), A(x_i,
\varepsilon_i, 2\varepsilon_i)))\leqslant
\frac{\omega_{n-1}}{I_i^{n-1}}<\infty\,,
\end{equation}
where $I_i=I_i(x_i, \varepsilon_i,
2\varepsilon_i)=\int\limits_{\varepsilon_i}^{2\varepsilon_i}
\frac{dr}{rq_{x_i}^{\,\prime\frac{1}{n-1}}(r)}$ and $I_i>0$ by the
construction.

\medskip
Finally, by~(\ref{eq12}), (\ref{eq13}), (\ref{eq7C}) and
(\ref{eq8C}) we obtain that
\begin{gather}M(\Gamma(f_m(A), E_m,\overline{{\Bbb
R}^n}))\leqslant M(\Gamma_m)\nonumber\\
\label{eq16} = M(f_m(\Delta_m))\leqslant
\sum\limits_{i=1}^NM(f_m(\Gamma(S(x_i, \varepsilon_i), S(x_i,
2\varepsilon_i), A(x_i,
\varepsilon_i, 2\varepsilon_i))))\\
\nonumber\leqslant\sum\limits_{i=1}^N\frac{\omega_{n-1}}{I_i^{n-1}}:=C<\infty\,.
\end{gather}
Since $P$ in~(\ref{eq9B}) may be done arbitrary large, the
relations~(\ref{eq9B}) and~(\ref{eq16}) contradict each other. This
completes the proof for the case when the relation~(\ref{eq2_4})
holds. Observe that, the case when $Q$ has a finite mean oscillation
in $A$ is a particular case of the relation~(\ref{eq2_4}) (see the
last part of the proof of Theorem~5.5 in \cite{Sev$_1$}). The proof
is wholly complete.~$\Box$
\end{proof}

\begin{corollary}\label{cor1}
{\it\, Assume that, under conditions of Lemma~\ref{lem4}, $A_0$ is
some another compactum in $D_0.$ Then there is $\delta_2>0$ and
$M_0\in {\Bbb N}$ such that $A_0\subset D_m$ for all $m\geqslant
M_0$ and $h(f_m(A_0),
\partial f_m(D_m))\geqslant \delta_2$ for every $m\geqslant M_0.$}
\end{corollary}

\medskip
\begin{proof}
Since $D_m$ converges to $D_0$ as its kernel, there is $M_0\in {\Bbb
N}$ such that $A_0\subset D_m$ for all $m\geqslant M_0.$ It remains
to show that $h(f_m(A_0),
\partial f_m(D_m))\geqslant \delta_2$ for every $m\geqslant M_0$ and
some $\delta_2>0.$  Assume the contrary, i.e., $h(f_m(A_0),
\partial f_m(D_m))\rightarrow 0$ as $m\rightarrow\infty.$ Now, $h(f_m(A_0),
\partial f_m(D_m))=h(x_m, y_m)\rightarrow 0$ as $m\rightarrow\infty$
for $x_m\in A_0$ and $y_m\in f_m(D_m).$ We may consider that
$x_m\rightarrow x_0$ as $m\rightarrow\infty,$ $x_0\in A_0.$ Put
$\varepsilon_0>0$ such that $\overline{B(x_0, \varepsilon_0)}\subset
D_0.$ We join $x_0$ with $A$ by a path $\gamma$ in $D_0.$ Now,
$B:=A\cup|\gamma|\cup \overline{B(x_0, \varepsilon_0)}$ is a
continuum in $D_0$ such that $B\subset D_m$ for sufficiently large
$m,$ $x_m\in B$ for sufficiently large $m$ and $h(f_m(B))\geqslant
h(f_m(B))\geqslant\delta$ for above $m.$ By Lemma~\ref{lem4},
$h(f_m(B),
\partial f_m(D_m))\geqslant \delta_*>0$ for $m\geqslant m_3,$ $m_3\in {\Bbb
N}.$ However, the latter contradicts with $h(f_m(A_0),
\partial f_m(D_m))=h(x_m, y_m)\rightarrow 0$ as $m\rightarrow\infty,$
because obviously $h(f_m(B),
\partial f_m(D_m))\leqslant h(x_m, y_m)\rightarrow 0$ as
$m\rightarrow\infty.$ The obtained contradiction completes the
proof.~$\Box$
\end{proof}

\section{Proof of Theorem~\ref{th2}}

\begin{proposition}\label{pr6}
{\it\, Let $D$ and $G$ be domains in ${\Bbb R}^n,$ $n\geqslant 2,$
such that $D\subset G,$ and let $K$ be a compactum in $D.$ Then
$h(K,
\partial D)\leqslant h(K, \partial G).$}
\end{proposition}

\medskip
\begin{proof}
Since $\overline{{\Bbb R}^n}$ is connected, $\partial
D\ne\varnothing.$ By the same reason, $\partial G\ne\varnothing.$
Thus, $h(K,
\partial G)$ and $h(K,
\partial D)$ are well-defined.

\medskip
Let now $h(K,
\partial D)=h(x_0, y_0),$ $x_0\in K$ and $y_0\in \partial D,$ and
let $h(K,
\partial G)=h(z_0, z_0),$ $z_0\in K$ and $w_0\in \partial G.$ We
need to prove that $h(x_0, y_0)\leqslant h(z_0, w_0).$ We prove this
by the contradiction: assume that $h(z_0, w_0)<h(x_0, y_0).$
Consider the ball $\overline{B_h(z_0, h(z_0, w_0))}.$ Since
$D\subset G,$ the point $w_0$ cannot be inner point of $D.$ Let us
join the points $z_0$ and $w_0$ by a path $\gamma$ in
$\overline{B_h(z_0, h(z_0, w_0))}.$ Now, $|\gamma|\cap
D\ne\varnothing\ne|\gamma|\setminus D.$ By Proposition~\ref{pr2}
there is $v_0\in \partial D\cap|\gamma|.$ Thus, $v_0\in
\overline{B_h(z_0, h(z_0, w_0))}$ that means $h(z_0, v_0)\leqslant
h(z_0, w_0).$ Since $z_0\in K$ and $v_0\in \partial D,$ the latter
means that $$h(K, \partial D)\leqslant h(z_0, v_0)\leqslant h(z_0,
w_0)<h(x_0, y_0)=h(K, \partial D)\,.$$
The obtained contradiction completes the proof.~$\Box$
\end{proof}

{\it Proof of Theorem~\ref{th2}}. \textbf{I}. Let us prove the
item~\textbf{I} of Theorem~\ref{th2} by the contradiction. Assume
that the assertion of Theorem~\ref{th2} is false. Then for some
$\varepsilon_0>0$ and for numbers $\delta_k=\frac{1}{k},$
$k=1,2,\ldots ,$ there are $m_k\in {\Bbb N}$ and $x_k, y_k\in
D_{m_k}$ such that $h(f_{m_k}(x_k), f_{m_k}(y_k))\geqslant
\varepsilon_0$ whenever $|x_k-y_k|<\frac{1}{k}.$ By Lemma~\ref{lem1}
every $f_{m_k}$ has a continuous extension to $\partial  D_{m_k},$
so that $f_{m_k}$ is uniformly continuous in $\overline{D_{m_k}}.$
Thus, we may consider that the sequence $m_k,$ $k=1,2,\ldots, $ is
increasing.  Without loss of generality, we may consider that the
above holds for $m=1,2, \ldots ,$ but not for $m_k.$ In particular,
\begin{equation}\label{eq2_1}
h(f_{m}(x_m), f_{m}(y_m))\geqslant \varepsilon_0\,,\qquad m\in {\Bbb
N}\,.
\end{equation}
By the compactness of $\overline{{\Bbb R}^n},$ we may consider that
$x_m, y_m\rightarrow x_0$ as $m\rightarrow\infty.$ There are two
cases: 1) $x_0\in D_0,$ 2) $x_0\in \partial D_0.$ In the first case,
when $x_0$ is the inner point of $D_0,$ all mappings $f_m$ except of
finite numbers are defined in some ball $B(x_0,
\varepsilon_0)\subset D_0\cap D_m,$ $m\geqslant M_0.$ Now, $f_m$ is
equicontinuous at $x_0$ (see \cite[Theorems~6.1, 6.5]{RS$_1$} for
$p=n$ and \cite[Theorems~4.1, 4.3]{Sev$_1$} for $n-1<p<n.$ Now, by
the triangle inequality,
\begin{equation}\label{eq1_1}h(f_m(x_m), f_m(y_m))\leqslant h(f_m(x_m),
f_m(x_0))+h(f_m(x_0), f_m(y_m))\,. \end{equation}
In this case, the right side in~(\ref{eq1_1}) tends to zero as
$m\rightarrow\infty$ which contradicts with~(\ref{eq2_1}), and the
proof finishes.

\medskip
Let us consider the case 2) $x_0\in \partial D_0.$ Observe that
$A_m:=A\subset {\Bbb R}^n\setminus B(x_0, \varepsilon_0)$ for some
$\varepsilon_0>0,$ because the $A$ is a continuum in $D_0.$ Thus,
all of conditions of Theorem~\ref{th1} are satisfied. By this
theorem, for any $\varepsilon>0$ there is $\delta=\delta(x_0,
\varepsilon)>0$ such that $h(f_m(x), f_m(x^{\,\prime}))<\varepsilon$
whenever $x, x^{\,\prime}\in B(x_0, \delta)\cap D_m$ and $m\in {\Bbb
N}.$ In particular, it follows from this theorem that there is
$M=M(\varepsilon)\in {\Bbb N}$ such that
\begin{equation}\label{eq1A_1}
h(f_m(x_m), f_m(y_m))<\varepsilon\,,\qquad m\geqslant
M(\varepsilon)\,.
\end{equation}
The relations~(\ref{eq2_1}) and~(\ref{eq1A_1}) contradict each other
and prove the uniform equicontinuity of~$f_m.$

\medskip
\textbf{II.} Given a compactum $K$ in $D_0,$ there is $m_0=m_0(K)$
such that $K\subset D_m$ for $m\geqslant m_0(K).$ Now, by the
proving above, $f_m,$ $m=1,2,\ldots,$ is uniformly equicontinuous
family on $K.$ By Arzela-Ascoli Theorem (see. e.g.,
\cite[Theorem~20.4]{Va}) $f_m$ forms a normal family on $K,$ i.e.,
there exists a subsequence $f_{m_k},$ $k=1,2,\ldots ,$ such that
$f_{m_k}(x)\rightarrow f(x)$ as $k\rightarrow\infty.$ Since $K$ is
an arbitrary compactum in $D_0,$ the mapping $f(x)$ is defined in
all $D_0.$ Now, we prove that $f$ has a continuous extension at
every point $x_0\in \partial D_0.$ By the proving above, given
$\varepsilon>0$ there is $\delta=\delta(\varepsilon)>0$ such that
$h(f_m(x), f_m(y))<\varepsilon$ whenever $x,y\in D_m,$
$|x-y|<\delta$ and $m\in {\Bbb N}.$ Fix $x_0\in \partial D_0.$ Let
$\delta_*=\delta_*(\varepsilon)=\delta/2$ and let $x, y\in B(x_0,
\delta_*)\cap D_0.$ Now, by the triangle inequality
$|x-y|<2\delta_*=\delta.$ Since $x, y\in D_0,$ there is $m_0=m_0(x,
y)\in {\Bbb N}$ such that $x, y\in D_0$ for every $m\geqslant m_0.$
Taking into account the mentioned above, we obtain that $h(f_m(x),
f_m(y))<\varepsilon,$ $m\geqslant m_0.$ Taking the limit as
$m\rightarrow\infty$ in the latter inequality, we obtain that
$h(f(x), f(y))\leqslant \varepsilon.$ Now, $f$ satisfies the Cauchy
condition at $x_0,$ so that $f$ has a limit at $x_0,$ as required.

\medskip
Let $A_0$ be a compactum in $D_0$ and let $G$ be an arbitrary domain
in $D_0$ with $A_0\cup A\overline{G}\subset D_0.$ Since $D_m$
converges to $D_0$ as its kernel, there is $m_0\in {\Bbb N}$ such
that $G\cup A_0\subset D_m,$ $m\geqslant m_0.$ Due to the above,
$f_{m_k}$ converges to $f$ uniformly in $G.$ Now, $f$ is either a
homeomorphism, or a constant in $G$ (see \cite[Theorems~4.1 and
4.2]{RS$_2$} for $p=n$ and \cite[Theorem~1]{Cr$_2$} for $n-1<p<n$).
Due to the condition $h(f_m(A))\geqslant\delta,$ $f$ is
homeomorphism in $G.$ Since $G$ is arbitrary compact subdomain in
$D_0,$ $f$ is continuous and injective in $D_0.$ Now, $f$ is a
homeomorphism in $D_0,$ as required (see, e.g.,
\cite[Corollary~3.1]{RS$_2$}).

\medskip
Let us to prove that, there is $\delta_2>0$ and $M_0\in {\Bbb N}$
such that $A_0\subset D_m$ for all $m\geqslant M_0$ and $h(f_m(A_0),
\partial f_m(D_m))\geqslant \delta_2$ for every $m\geqslant M_0.$
Without loss of generality, we may consider that $f_m$ converges to
$f$ locally uniformly as $m\rightarrow\infty.$ Let
$\varepsilon:=\frac{1}{2}h(f(A_0),
\partial f(G))$ ($\varepsilon>0$ because $f$ is a
homeomorphism in $G$ and, consequently, $f(A_0)\in
\textrm{Int}\,f(G)$). By the triangle inequality and due to the
uniform convergence of $f_m$ to $f,$ given $\varepsilon>0$ there is
$M=M(\varepsilon)$ such that $h(f_m(x), f(x))<\varepsilon/2$ for any
$x\in A_0.$ Let $y\in \partial f(G).$ By the triangle inequality
$$h(f_m(x), y)\geqslant h(y, f(x))-h(f(x), f_m(x))\geqslant \varepsilon-\varepsilon/2=\varepsilon/2
=\frac{1}{4}h(f(A_0),
\partial f(G))\,.$$
Thus, taking the $\inf$ over $x\in A_0$ and $y\in \partial f(G),$ we
obtain that $h(f_m(A_0), \partial f(G))\geqslant
\frac{1}{4}h(f(A_0),
\partial f(G)):=\delta_*.$ Let now $$h(f_m(A_0), \partial f_m(G))=h(f_m(A_0), f_m(\partial
G))=h(f_m(x_m), f_m(x^{\,\prime}_m))\,,$$ where $x_m\in A_0$ and
$x^{\,\prime}_m\in \partial G.$ Now, by the triangle inequality and
by the locally uniform convergence of $f_m$ to $f,$ we obtain that
$$h(f_m(x_m), f_m(x^{\,\prime}_m))\geqslant h(f_m(x_m), f(x^{\,\prime}_m))-
h(f(x^{\,\prime}_m), f_m(x^{\,\prime}_m))\geqslant
\delta_*-\delta_*/2=\delta_*/2>0$$
for sufficiently large $m\in {\Bbb N}.$ The latter inequality
implies that $h(f_m(A_0), \partial f_m(G))\geqslant
\delta_*/2:=\delta_2.$ Finally, since $G\subset D_m$ for
sufficiently large $m\in {\Bbb N},$ by Proposition~\ref{pr6} we
obtain that $h(f_m(A_0), \partial f_m(G))\geqslant\delta_2$ for
above $m,$ as required.

\textbf{III.} Let $x_m\in D_{m},$ $m=1,2,\ldots,$ $x_m\rightarrow
x_0$ and $f_m\rightarrow f$ locally uniformly in $D_0$ as
$m\rightarrow \infty.$ We need to show that $f_{m}(x_m)\rightarrow
f(x_0).$ If $x_0\in D_0,$ the relation $f_{m}(x_m)\rightarrow
f(x_0)$ as $m\rightarrow\infty$ directly follows from the locally
uniform convergence of $f_m$ to $f$ in $D_0.$ Now, let $x_0\in
\partial D_0.$ We lead the proof by the contradiction, i.e., assume
that $h(f_{m_k}(x_{m_k}), f(x_0))\geqslant \varepsilon_0>0$ for some
an increasing sequence of numbers $m_k=1,2,\ldots ,$ and some
$\varepsilon_0>0.$ Again, going to renumbering, if required, we may
assume that
\begin{equation}\label{eq1A}
h(f_{m}(x_{m}), f(x_0))\geqslant \varepsilon_0>0\,,\qquad
m=1,2,\ldots \,.
\end{equation}
Consider some a sequence $y_m\rightarrow x_0,$ $y_m\in D_0,$
$m\rightarrow\infty.$ Now, given $m\in {\Bbb N},$ by the locally
uniform convergence of $f_m$ to $f,$ there is $k=k_m$ such that
$h(f_{k_m}(y_{k_m}), f(y_{k_m}))<1/m\,.$ We may consider that $k_m$
is an increasing sequence by $m=1,2,\ldots .$ Let $z_m:=y_{k_m}.$
Now, $z_m\rightarrow x_0$ and $f(z_m)\rightarrow f(z_0)$ as
$m\rightarrow \infty$ because $f$ is continuous at $x_0$ (see
item~\textbf{II}). Now, by the triangle inequality and~\textbf{I} we
obtain that
\begin{gather}h(f_{k_m}(x_{k_m}), f(x_0))\leqslant h(f_{k_m}(x_{k_m}),
f_{k_m}(z_m))\nonumber\\
\label{eq1}+h(f_{k_m}(z_{m}), f(z_m))+h(f(z_m),f(x_0))\rightarrow 0
\end{gather}
as $m\rightarrow\infty.$ The relation~(\ref{eq1}) contradicts
with~(\ref{eq1A}), that proves that $f_{m}(x_m)\rightarrow f(x_0)$
as $m\rightarrow\infty.$ Finally, $f(x_0)\in \partial f(D_0)$
whenever $x_0\in
\partial D_0,$ because $f$ is a homeomorphism in $D_0$ (see, e.g.,
\cite[Proposition~13.5]{MRSY}). Theorem is completely proved.~$\Box$

\section{Main Lemma for mappings with a branching}

Given $p\geqslant 1,$ a number $\delta>0,$ a family $\frak{D}$ of
domains in ${\Bbb R}^n,$ $n\geqslant 2,$ and a Lebesgue measurable
function $Q:{\Bbb R}^n\rightarrow[0, \infty]$, we denote by
$\frak{J}_{Q, \delta, p}(\frak{D})$ the family of all open discrete
closed ring mappings $f:D_f\rightarrow \overline{{\Bbb R}^n}$
defined in some a domain $D_f\in \frak{D}$ and
satisfying~(\ref{eq2*!})--(\ref{eq8B_2}) at any point
$x_0\in\overline{\frak{D}}$ such that the following condition holds:
for any domain $D^{\,\prime}_f:=f(D_f)$ there exists a continuum
$K_f\subset D^{\,\prime}_f$ such that $h(K_f)\geqslant \delta$ and
$h(f^{\,-1}(K_f), \partial D_f)\geqslant \delta>0.$ The following
lemma is an analog of Lemma~\ref{lem2} for mappings which are not
homeomorphisms.

\medskip
\begin{lemma}\label{lem3}
{\it\, Let $p\in (n-1, n],$ let $x_0\in \overline{\frak{D}}$ and let
$f_m\in \frak{J}_{Q, \delta, p}(\frak{D}),$ $m=1,2,\ldots ,$
$f_m:D_m\rightarrow \overline{{\Bbb R}^n}$ be a sequence such that:

\medskip
1) for every $m\in {\Bbb N},$ a domain $D_m$ is locally connected at
$\partial D_m;$

2) for every neighborhood $U$ of $x_0$ there is a neighborhood $V$
of $x_0,$ $V\subset U,$ and $M=M(U, x_0)\in {\Bbb N}$ such that
$V\cap D_m$ is connected for every $m\geqslant M(U, x_0);$

3) there is $\varepsilon_0>0$ such that $f^{\,-1}_m(K_m)\subset
{\Bbb R}^n\setminus B(x_0, \varepsilon_0),$ where $K_m:=K_{f_m}$ is
a continuum corresponding to the definition of the class
$\frak{J}_{Q, \delta, p}(\frak{D})$ for $f_m;$

4) the family $f_m(D_m)$ is equi-uniform with respect to $p$-modulus
over all $m\in {\Bbb N};$

5) there exists $\varepsilon_0=\varepsilon_0(x_0)>0$ and a Lebesgue
measurable function $\psi:(0, \varepsilon_0)\rightarrow [0,\infty]$
such that relations~(\ref{eq7***_2})--(\ref{eq3.7.2_2}) hold
as $\varepsilon\rightarrow 0,$ where $A(x_0, \varepsilon,
\varepsilon_0)$ is defined~in~(\ref{eq1**}).

\medskip
Then the family $f_m,$ $m=1,2,\ldots,$ is equicontinuous with
respect to $x_0,$ i.e., for any $\varepsilon>0$ there is
$\delta=\delta(x_0, \varepsilon)>0$ such that $h(f_m(x),
f_m(x^{\,\prime}))<\varepsilon$ whenever $x, x^{\,\prime}\in B(x_0,
\delta)\cap D_m$ and $m\in {\Bbb N}.$}
\end{lemma}

\medskip
\begin{proof} Suppose the contrary. Then there exists $x_0\in \overline{\frak{D}}$ and
a number $a>0$ with the following property: for $\delta=1/k,$
$k=1,2,\ldots,$ there are $m_k\in {\Bbb N}$ and $x_k,
x^{\,\prime}_k\in B(x_0, 1/k)\cap D_{m_k}$ such that
$h(f_{m_k}(x_k), f_{m_k}(x^{\,\prime}_k))\geqslant a.$ Arguing
similarly to the proof of Lemma~\ref{lem2}, we may consider that the
sequence $m_k$ is increasing by $k.$ Without loss of generality,
going to renumbering, if required, we may consider that
$h(f_{m_k}(x_k), f_{m_k}(x^{\,\prime}_k))\geqslant a$ holds for $m$
instead of $m_k,$ i.e., (\ref{eq6***}) is satisfied. We may assume
that $x_0\ne \infty.$ Now, we use the condition 2): for every
neighborhood $U$ of $x_0$ there is a neighborhood $V$ of $x_0,$
$V\subset U,$ and $M=M(k)$ such that $V\cap D_m$ is connected for
every $m\geqslant M(k).$ Thus, given a ball $U_1=B(x_0, 2^{\,-1})$
there is $V_1\subset U_1$ and $M(1)\in {\Bbb N}$ such that $D_m\cap
V_1$ is connected for $m\geqslant M(1).$ Since $x_m,
x_m^{\,\prime}\rightarrow x_0$ as $m\rightarrow \infty,$ there is
$m_1\geqslant M(1)$ such that $x_{m_1}, x^{\,\prime}_{m_1}\in
V_1\cap D_{m_1}.$ Follow, given a ball $U_2=B(x_0, 2^{\,-2})$ there
is $V_2\subset U_2$ and $M(2)\in {\Bbb N}$ such that $D_m\cap V_2$
is connected for $m\geqslant M(2).$ Since $x_m,
x_m^{\,\prime}\rightarrow x_0$ as $m\rightarrow \infty,$ there is
$m_2\geqslant \max\{M(2), m_1\}$ such that $x_{m_2},
x^{\,\prime}_{m_2}\in V_2\cap D_{m_2}.$ Etc. In general, given a
ball $U_k=B(x_0, 2^{\,-k})$ there is $V_k\subset U_k$ and $M(k)\in
{\Bbb N}$ such that $D_m\cap V_k$ is connected for $m\geqslant
M(k).$ Since $x_m, x_m^{\,\prime}\rightarrow x_0$ as $m\rightarrow
\infty,$ there is $m_k\geqslant \max\{M(k), m_{k-1}\}$ such that
$x_{m_k}, x^{\,\prime}_{m_k}\in V_k\cap D_{m_k}.$ Relabeling, if
required, we may consider that the same sequences $x_m$ and
$x^{\prime}_m$ satisfy the above conditions, i.e., given a ball
$U_m=B(x_0, 2^{\,-m}),$ $m=1,2,\ldots,$ there is $V_m\subset U_m$
such that $x_m, x^{\,\prime}_m\in D_m\cap V_m$ while $D_m\cap V_m$
is connected. We join the points $x_m$ and $x^{\,\prime}_m$ by a
path $\gamma_m:[0,1]\rightarrow D_m\cap V_m$ such that
$\gamma_m(0)=x_m,$ $\gamma_m(1)=x^{\,\prime}_m$ and $\gamma_m(t)\in
D_m\cap V_m$ for $t\in (0,1),$ see~Figure~\ref{fig2}.
\begin{figure}[h]
\centerline{\includegraphics[scale=0.5]{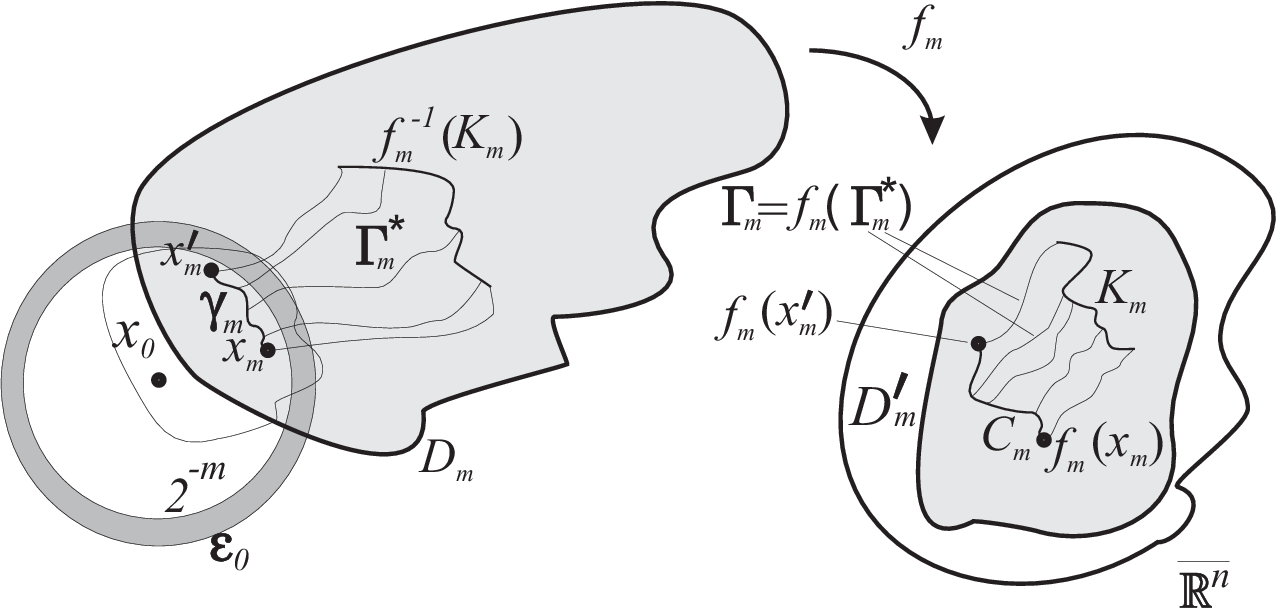}} \caption{To
prove Lemma~\ref{lem3}}\label{fig2}
\end{figure}
Denote by $C_m$ the image of the path $\gamma_m$ under the mapping
$f_m.$ It follows from relation (\ref{eq6***}) that a
condition~(\ref{eq5.1}) is satisfied, where $h(C_m)$ denotes the
chordal diameter of the set $C_m.$

By the definition of the family $\frak{J}_{Q, \delta, p}(\frak{D}),$
for any $f_m\in\frak{J}_{Q, \delta, p}(\frak{D})$ and any domain
$D^{\,\prime}_m:=f_m(D_m)$ there exists a continuum $K_m\subset
D^{\,\prime}_m$ such that $h(K_m)\geqslant \delta$ and
$h(f_m^{\,-1}(K_m),
\partial D_m)\geqslant \delta>0.$
Since by the condition of the lemma the domains $D^{\,\prime}_m$
form the equi-uniform family with respect to $p$-modulus and due to
(\ref{eq5.1}), we obtain that the inequality
\begin{equation}\label{eq13_4}
M_p(\Gamma(K_m, C_m, D^{\,\prime}_m))\geqslant b
\end{equation}
holds for all $m=1,2,\ldots$ and some $b>0.$
Consider the family $\Gamma_m,$ consisting of all paths $\beta:[0,
1)\rightarrow D^{\,\prime}_m,$ where $\beta(0)\in C_m$ and
$\beta(t)\rightarrow p\in K_m$ as $t\rightarrow 1.$ Let $\Gamma^*_m$
be a family of all total $f_m$-liftings $\alpha:[0, 1)\rightarrow D$
of the family $\Gamma_m$ starting at $|\gamma_m|$ (see
Proposition~\ref{pr4_a}). Since $f_m$ is closed, we have:
$\alpha(t)\rightarrow f^{\,-1}_m(K_m),$ where $f^{\,-1}_m(K_m)$ is
the total preimage of the continuum $K_m$ under the mapping $f_m.$
By~(\ref{eq2*!}), it follows that
\begin{gather}
\nonumber M_p(f_m(\Gamma_m^*))\\
\label{eq10A_2} \leqslant M_p(f_m(\Gamma(|\gamma_m|,
f^{\,-1}_m(K_m), D_m)))\leqslant \int\limits_{A(x_0, 2^{\,-m},
\varepsilon_0)} Q(x)\cdot \eta^p(|x-x_0|)\ dm(x)
\end{gather}
for each measurable function $\eta: (2^{\,-m},
\varepsilon_0)\rightarrow [0,\infty ],$ such that
$\int\limits_{2^{\,-m}}^{\varepsilon_0}\eta(r)\,dr \geqslant 1.$
Note that, the function
$$\eta(t)=\left\{
\begin{array}{rr}
\psi(t)/I(2^{\,-m}, \varepsilon_0), & t\in (2^{\,-m},
\varepsilon_0),\\
0, & t\in {\Bbb R}\setminus (2^{\,-m}, \varepsilon_0)\,,
\end{array}
\right. $$ where $I(\varepsilon,
\varepsilon_0):=\int\limits_{\varepsilon}^{\varepsilon_0}\psi(t)\,dt,$
satisfies the normalization condition~(\ref{eq8B_2}). Therefore from
the conditions~(\ref{eq7***_2})--(\ref{eq3.7.2_2}) and
(\ref{eq10A_2}) it follows that
\begin{equation}\label{eq11B_1}
M_p(f_m(\Gamma^{\,*}_m))\leqslant \alpha(2^{\,-m})\rightarrow 0
\end{equation}
as $m\rightarrow \infty,$ where $\alpha(\varepsilon)$ is some
nonnegative function tending to zero as $\varepsilon\rightarrow 0,$
which exists by the condition (\ref{eq3.7.2_2}). In addition,
$f_m(\Gamma^{\,*}_m)= \Gamma_m$ and $M_p(\Gamma_m)=M_p(\Gamma(K_m,
C_m, D^{\,\prime}_m)).$ Thus
\begin{equation}\label{eq12_4}
M_p(f_m(\Gamma^{\,*}_m))=M_p(\Gamma(K_m, C_m, D^{\,\prime}_m))\,.
\end{equation}
However, the relation~(\ref{eq11B_1}) contradicts with
(\ref{eq13_4}) and (\ref{eq12_4}). The resulting contradiction
indicates that the assumption (\ref{eq6***}) was incorrect, and,
therefore, the family $\frak{J}_{Q, \delta, p}(\frak{D})$ is
equicontinuous at $x_0\in
\partial D.$~$\Box$
\end{proof}

\medskip
\begin{theorem}\label{th3}
{\it\, Let $p\in (n-1, n],$ let $x_0\in \overline{\frak{D}}$ and let
$f_m\in \frak{J}_{Q, \delta, p}(\frak{D}),$ $m=1,2,\ldots ,$
$f_m:D_m\rightarrow \overline{{\Bbb R}^n}$ be a sequence such that:

\medskip
1) for every $m\in {\Bbb N},$ a domain $D_m$ is locally connected at
$\partial D_m;$

2) for every neighborhood $U$ of $x_0$ there is a neighborhood $V$
of $x_0,$ $V\subset U,$ and $M=M(U, x_0)\in {\Bbb N}$ such that
$V\cap D_m$ is connected for every $m\geqslant M(U, x_0);$

3) there is $\varepsilon_0>0$ such that $f^{\,-1}_m(K_m)\subset
{\Bbb R}^n\setminus B(x_0, \varepsilon_0),$ where $K_m:=K_{f_m}$ is
a continuum corresponding to the definition of the class
$\frak{J}_{Q, \delta, p}(\frak{D})$ for $f_m;$

4) the family $f_m(D_m)$ is equi-uniform with respect to $p$-modulus
over all $m\in {\Bbb N};$

5) at least one of two following conditions hold: $Q$ has a finite
mean oscillation in $\overline{D_0},$ or relation~(\ref{eq2_4})
holds for every $x_0\in \overline{D_0}$ and some $\beta(x_0)>0.$

\medskip
Then the family $f_m,$ $m=1,2,\ldots,$ is equicontinuous with
respect to $x_0,$ i.e., for any $\varepsilon>0$ there is
$\delta=\delta(x_0, \varepsilon)>0$ such that $h(f_m(x),
f_m(x^{\,\prime}))<\varepsilon$ whenever $x, x^{\,\prime}\in B(x_0,
\delta)\cap D_m$ and $m\in {\Bbb N}.$}
\end{theorem}

\medskip
{\it Proof} follows directly by Lemma~\ref{lem3} and
Proposition~\ref{pr1}.~$\Box$

\medskip
{\it Proof of Theorem~\ref{th4_4}.} The proofs of the statements in
item~\textbf{I} are based on Theorem~\ref{th3} (instead of
Theorem~\ref{th1}) and are completely similar to the corresponding
proofs of Theorem~\ref{th2}.

\medskip
\textbf{II.} Let $A_0$ be a compactum in $D_0.$ We show that there
is $\delta_2>0$ and $M_0\in {\Bbb N}$ such that $A_0\subset D_m$ and
$h(f_m(A_0),
\partial f_m(D_m))\geqslant \delta_2$ for every $m\geqslant M_0.$
Note that, the relation $A_0\subset D_m$ for sufficiently large
$m\in{\Bbb N}$ holds obviously, because $D_m$ converges to $D_0$ as
its kernel by the assumption. Now, we show that $h(f_m(A_0),
\partial f_m(D_m))\geqslant \delta_2$
for sufficiently large $m\in {\Bbb N}.$ Assume the contrary. Then
$h(f_{m_k}(A_0),
\partial f_{m_k}(D_{m_k}))\rightarrow 0$ as $k\rightarrow\infty$
for some increasing a sequence of
numbers $m_k,$ $k=1,2,\ldots .$ Going to renumbering, if required,
we may consider that the latter holds for $m$ instead of $m_k,$
i.e.,
\begin{equation}\label{eq2}
h(f_{m}(A_0),
\partial f_{m}(D_{m}))\rightarrow 0\,,\quad m\rightarrow\infty\,.
\end{equation}
Let $K_m\subset D^{\,\prime}_m$ be a continuum from the definition
of the class $\frak{R}_{Q, \delta, p, E}(D_0, \frak{D})$ such that
$h(K_m)\geqslant \delta$ and $h(f_m^{\,-1}(K_m),
\partial D_m)\geqslant \delta>0.$ Let $h(K_m)=h(y_m, z_m),$ where $y_m, z_m\in
K_m.$ Now, $y_m=f_m(x_m)$ and $z_m=f_m(x^{\,\prime}_m)$ for some
$x_m, x^{\,\prime}_m\in f_m^{\,-1}(K_m).$ We may consider that
$x_m\rightarrow x_0$ and $x^{\,\prime}_m\rightarrow x^{\,\prime}_0$
as $m\rightarrow\infty.$ Due to the condition $h(f_m^{\,-1}(K_m),
\partial D_m)\geqslant \delta>0$ we have that $h(x_m, \partial D_m)\geqslant
\delta$ and $h(x^{\,\prime}_m, \partial D_m)\geqslant \delta,$
$m=1,2,\ldots .$ Since by the assumption $D_m\subset D_0,$
$m=1,2,\ldots,$ by Proposition~\ref{pr6} $h(x_m, \partial
D_0)\geqslant \delta$ and $h(x^{\,\prime}_m, \partial D_0)\geqslant
\delta,$ $m=1,2,\ldots .$ Now, $x_0, x^{\,\prime}_0\in D_0.$ Let
$\sigma_1, \sigma_2>0$ be numbers such that $\overline{B(x_0,
\sigma_1)}\subset D_0$ and $\overline{B(x^{\,\prime}_0,
\sigma_2)}\subset D_0.$ We may consider that $x_m\in
\overline{B(x_0, \sigma_1)}$ and $x^{\,\prime}_m\in
\overline{B(x^{\,\prime}_0, \sigma_2)}$ for any $m\in {\Bbb N}.$
Now, we join the points $x_0$ and $x^{\,\prime}_0$ by a path
$\gamma$ in $D_0.$ Let $A:=\overline{B(x_0,
\sigma_1)}\cup|\gamma|\cup \overline{B(x^{\,\prime}_0, \sigma_2)},$
now $A$ is a continuum in $D_0.$ Since $D_m$ converges to $D_0$ as
its kernel, we may consider that $A\subset D_m$ for any
$m=1,2,\ldots .$ Observe that, $h(f_m(A))\geqslant h(f_m(x_m),
f_m(x^{\,\prime}_m))=h(y_m, z_m)\geqslant \delta>0.$ Now, by
Lemma~\ref{lem4} $h(f_{m}(A),
\partial f_{m}(D_{m}))\geqslant \delta_*>0$ for all $m\in {\Bbb N}$
and some $\delta_*>0.$ Besides that, by Corollary~\ref{cor1}
$h(f_{m}(A_0),
\partial f_{m}(D_{m}))\geqslant \delta_{**}>0$ for all $m\in {\Bbb N}$
and some $\delta_{**}>0.$ The latter contradicts with the
assumption~(\ref{eq2}). Thus, $h(f_m(A_0),
\partial f_m(D_m))\geqslant \delta_2$ for sufficiently large $m\in{\Bbb
N},$ as required.

\medskip Let us to prove that, the limit mapping $f$ is boundary
preserving. In other words, wee need to prove that, if $x_0\in
\partial D_0,$ then $f(x_0)\in \partial f(D_0).$
Assume the contrary: let $x_0\in \partial D_0,$ however, $f(x_0)\in
{\rm Int\,}f(D_0).$ Now, $\partial D_m\cap U\ne\varnothing$ for any
neighborhood $U$ of $x_0$ and for infinitely many $m=1,2,\ldots .$
So, we may construct a sequence $x_k\in \partial D_{m_k},$
$k=1,2,\ldots,$ such that $x_k\rightarrow x_0$ as
$k\rightarrow\infty.$ For simplicity, we may consider that $x_m\in
\partial D_m$ and $x_m\rightarrow x_0$ as
$m\rightarrow\infty.$ Since $f_m$ are closed, they are boundary
preserving (see, e.g., \cite[Theorem~3.3]{Vu}). Thus, for any $m\in
{\Bbb N}$ we may find $y_m\in D_m$ such that $|x_m-y_m|<\frac{1}{m}$
and $h(f_m(y_m), \partial f_m(D_m))<\frac{1}{m}.$ Let $h(f_m(y_m),
\partial f_m(D_m))=h(f_m(y_m), z_m).$ Let $A_0:=\{y_0\}$ and $y_0$ is some point
of $D_0$ with $f(y_0)=f(x_0).$ Now, $y_0\in D_m$ for all $m\in {\Bbb
N}$ except a finite number. By the triangle inequality and proving
above we obtain that
\begin{gather*}
h(f_m(y_0), z_m)\leqslant h(f_m(y_0), f(x_0))+ h(f(x_0),
f_m(y_m))+h(f_m(y_m), z_m)\rightarrow 0\end{gather*}
as $m\rightarrow\infty.$ The latter contradicts with the statement
that $h(f_m(A_0),
\partial f_m(D_m))\geqslant \delta_2$ for every $m\geqslant M_0,$

\medskip
Let us to prove the relation~(\ref{eq3A}). We lead the proof by the
contradiction, i.e., assume that (\ref{eq3A}) does not hold. Now,
given $k\in {\Bbb N}$ there is $m=m_k\in {\Bbb N}$ such that
$B_h(f_{m_k}(x_0), 1/k)\setminus f_{m_k}(D_{m_k})\ne \varnothing$
whenever $\overline{B(x_0, \varepsilon_0)}\subset
D_0\cap\bigcap\limits_{m=1}^{\infty}D_m.$ Let $p_k\in
B_h(f_{m_k}(x_0), 1/k)\setminus f_{m_k}(D_{m_k}).$ We join the
points $f_{m_k}(x_0)$ and $p_k$ inside the ball $B_h(f_{m_k}(x_0),
1/k)$ by a path $\gamma_k.$ Now, $|\gamma_k|\cap
f_{m_k}(D_{m_k})\ne\varnothing \ne |\gamma_k|\setminus
f_{m_k}(D_{m_k})$ and by Proposition~\ref{pr2} $|\gamma_k|\cap
\partial f_{m_k}(D_{m_k})\ne\varnothing.$ Thus, $h(f_{m_k}(x_0),
\partial f_{m_k}(D_{m_k}))\rightarrow 0$ as $k\rightarrow \infty,$ that
contradicts with $h(f_m(A_0),
\partial f_m(D_m))\geqslant \delta_2$ for every $m\geqslant M_0,$
proved above. Theorem is proved.~$\Box$

\section{The versions for prime ends}

The following statement may be proved completely similarly
to~\cite[Lemma~5.2]{Sev$_1$}, cf.~\cite[Theorem~5, Example~5]{Ad},
\cite[Lemma~5, Theorem~3]{KR}.

\medskip
\begin{lemma}\label{lem6}{\it\, Let $p\geqslant 1,$ let $D$ be a
quasiconformally regular domain, let $f:D\rightarrow \overline{{\Bbb
R}^n}$ be an open discrete and closed ring $Q$-mapping with respect
to the $p$-modulus at all points $b\in
\partial D,$ let
$f(D)=D^{\,\prime},$ and let $D^{\,\prime}$ be a domain with a
strongly accessible boundary with a respect to $p$-modulus at least
at one point $y\in C(f, b).$ Assume that, for any $b\in
\partial D$ there is $\varepsilon_0>0$ and some positive measurable
function $\psi:(0, \varepsilon_0)\rightarrow (0,\infty)$ such that
\begin{equation*}\label{eq7***_A}
0<I(\varepsilon,
\varepsilon_0)=\int\limits_{\varepsilon}^{\varepsilon_0}\psi(t)\,dt
< \infty
\end{equation*}
for any $\varepsilon\in(0, \varepsilon_0)$ and, in addition,
\begin{equation*}\label{eq5***_A}
\int\limits_{A(b, \varepsilon, \varepsilon_0)}
Q(x)\cdot\psi^{\,p}(|x-b|)
 \ dm(x) =o(I^p(\varepsilon, \varepsilon_0))\,,
\end{equation*}
where $A:=A(b, \varepsilon, \varepsilon_0)$ is defined
in~(\ref{eq1**}). Then $f$ have a continuous extension
$\overline{f}:\overline{D}_P\rightarrow \overline{D^{\,\prime}}$
such that $f(\overline{D}_P)=\overline{D^{\,\prime}}.$}
\end{lemma}

\medskip
The following statement holds, cf.~\cite{SevSkv}.

\medskip
\begin{lemma}\label{lem2_A}
{\it\, The statement of Theorem~\ref{th4} remains true, if, under
assumptions of Theorem~\ref{th4}, instead of the assumption~4) we
require the following: there exists
$\varepsilon_0=\varepsilon_0(x_0)>0$ and a Lebesgue measurable
function $\psi:(0, \varepsilon_0)\rightarrow [0,\infty]$ such
that~(\ref{eq7***_2}) holds while the relation~(\ref{eq3.7.2_2})
holds as $\varepsilon\rightarrow 0,$ where $A(x_0, \varepsilon,
\varepsilon_0)$ is defined~in (\ref{eq1**}).}
\end{lemma}

\medskip
\begin{proof}
\textbf{I.} Assume the contrary. Then there exists $a>0$ such that
for $\delta=1/k,$ $k=1,2,\ldots,$ there are $m_k\in {\Bbb N}$ and
$x_k, x^{\,\prime}_k\in D_{m_k}$ such that $\rho(x_k,
x^{\,\prime}_k)<1/k,$ however, $h(f_{m_k}(x_k),
f_{m_k}(x^{\,\prime}_k))\geqslant a.$ By the definition,
$\overline{D_0}_P$ is a compact space, so that we may consider that
the sequences $x_k, x^{\,\prime}_k\rightarrow P_0\in
\overline{D_0}_P$ by the metrics $\rho.$ We may consider that the
sequence $m_k$ is increasing by $k.$ Otherwise, $m_k=k_0$ for
sufficiently large $k\in{\Bbb N}$ and some $k_0\in {\Bbb N},$
besides that, $P_0\in \overline{D_{k_0}}_P$ because $f_{k_0}$ is
defined at $x_k\cap D_{k_0}$ and $x_k\rightarrow P_0$ as
$k\rightarrow\infty.$ Now, from the assumption above we have that
$h(f_{k_0}(x_k), f_{k_0}(x^{\,\prime}_k))\geqslant a.$ The latter
contradicts with Lemma~\ref{lem6}. So, $m_k$ may be chosen as an
increasing sequence. Without loss of generality, going to
renumbering, if required, we may consider that the relation
$h(f_{m_k}(x_k), f_{m_k}(x^{\,\prime}_k))\geqslant a$ holds for $m$
instead of $m_k,$ i.e., $x_m, x^{\,\prime}_m\in D_m$ and
\begin{equation}\label{eq6***_A}
h(f_m(x_m), f_m(x^{\,\prime}_m))\geqslant a/2\qquad \forall\,\,m\in
{\Bbb N}\,.
\end{equation}
If $P_0\in D_0,$ then all mappings $f_m$ except of finite numbers
are defined in some ball $B(P_0, \varepsilon_0)\subset D_0\cap D_m,$
$m\geqslant M_0.$ Now, $f_m$ is equicontinuous at $P_0$ (see
\cite[Lemma~4.9]{RS$_1$} for $p=n$ and \cite[Lemma~4.3]{Sev$_1$} for
$n-1<p<n.$ Therefore, we may consider that $P_0\in
E_{D_0}=\overline{D_0}_P\setminus D_0.$

\medskip
Let $d_m$ and $\sigma_m,$ $m=0,1,2,\ldots, $ be sequences of domains
and cuts in $P_0,$ respectively, while $\sigma_m$ lie on the spheres
$S(z_0, r_m)$ centered at some point $z_0\in
\partial D_0,$ where $r_m\rightarrow 0$ as $m\rightarrow\infty$
(such a sequence $\sigma_m$ exists by~\cite[Lemma~3.1]{IS},
cf.~\cite{KR}). Since $D_m$ is a regular sequence of domains with a
respect to $D_0$ in the terms of prime ends, there exists a sequence
of cuts $\varsigma_{m}$ and domains $d^{\,\prime}_m,$ equivalent to
$\sigma_m$ and $d_m,$ respectively, such that the following
condition holds: given $d^{\,\prime}_k,$ $k=1,2,\ldots,$ there is
$M_1=M_1(k)$ such that $d^{\,\prime}_k\cap D_m$ is a non-empty
connected set for every $m\geqslant M_1(k).$ Now, we will consider
that $d^{\,\prime}_k\subset d_k$ for any $k=1,2,\ldots .$

\medskip
Thus, given a domain $d^{\,\prime}_1$ there is a number $M(1)\in
{\Bbb N}$ such that $D_m\cap d^{\,\prime}_1$ is connected for
$m\geqslant M(1).$ Since $x_m, x_m^{\,\prime}\rightarrow P_0$ as
$m\rightarrow \infty,$ there is $m_1\geqslant M(1)$ such that
$x_{m_1}, x^{\,\prime}_{m_1}\in d^{\,\prime}_1\cap D_{m_1}.$ Follow,
given a domain $d^{\,\prime}_2$ there is a number $M(2)\in {\Bbb N}$
such that $D_m\cap d^{\,\prime}_2$ is connected for $m\geqslant
M(2).$ Since $x_m, x_m^{\,\prime}\rightarrow P_0$ as $m\rightarrow
\infty,$ there is $m_2\geqslant M(2)$ such that $x_{m_2},
x^{\,\prime}_{m_2}\in d^{\,\prime}_2\cap D_{m_2}.$ Etc. In general,
given $k\in {\Bbb N}$ and a domain $d^{\,\prime}_k$ there is a
number $M(k)\in {\Bbb N}$ such that $D_m\cap d^{\,\prime}_k$ is
connected for $m\geqslant M(k).$ Since $x_m,
x_m^{\,\prime}\rightarrow P_0$ as $m\rightarrow \infty,$ there is
$m_k\geqslant M(k)$ such that $x_{m_k}, x^{\,\prime}_{m_k}\in
d^{\,\prime}_k\cap D_{m_k}.$ Relabeling, if required, we may
consider that the same sequences $x_m$ and $x^{\,\prime}_m$ satisfy
the above conditions, i.e., $x_{m}, x^{\,\prime}_{m}\in
d^{\,\prime}_m\cap D_{m}$ while $d^{\,\prime}_m\cap D_{m}$ is
connected. We join the points $x_m$ and $x^{\,\prime}_m$ by a path
$\gamma_m:[0,1]\rightarrow D_m\cap d^{\,\prime}_m$ such that
$\gamma_m(0)=x_m,$ $\gamma_m(1)=x^{\,\prime}_m$ and $\gamma_m(t)\in
D_m\cap d^{\,\prime}_m$ for $t\in (0,1),$ see~Figure~\ref{fig6_2}.
\begin{figure}[h]
\centerline{\includegraphics[scale=0.5]{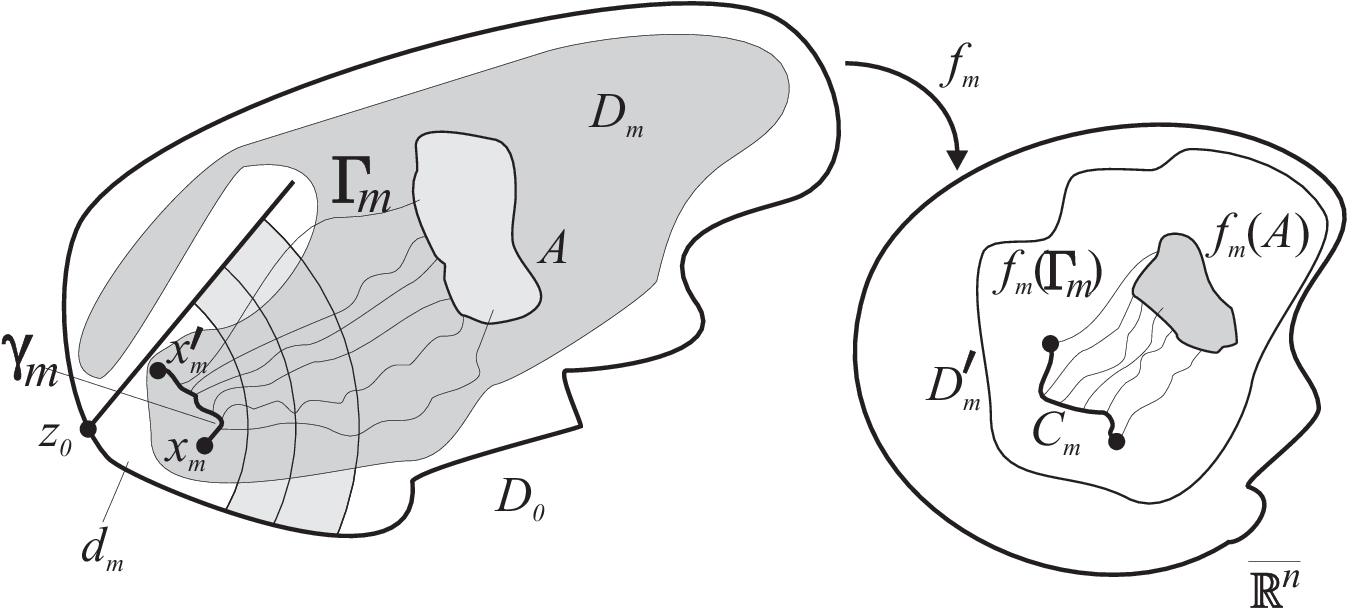}} \caption{To
prove Lemma~\ref{lem1}}\label{fig6_2}
\end{figure}
Denote by $C_m$ the image of the path $\gamma_m(t)$ under $f_m.$ It
follows from relation (\ref{eq6***_A}) that
\begin{equation}\label{eq5.1_A}
h(C_m)\geqslant a/2\qquad\forall\, m\in {\Bbb N}\,,
\end{equation}
where $h$ denotes the chordal diameter of the set.

\medskip
Let $\Gamma_m$ be a family of paths joining $\gamma_m$ and $A$ in
$D_m.$ Observe that, $A\subset D_0\setminus d_m$ for sufficiently
large $m\in {\Bbb N}.$ We may consider that the latter holds for all
$m\in {\Bbb N}.$ Now, by Proposition~\ref{pr2}
\begin{equation}\label{eq2A}
\Gamma_m>\Gamma(S(z_0, r_m), S(z_0, r_1), D_0)\,.
\end{equation}
By the definition of the class $\frak{F}_{Q, A, p, \delta}(D_0,
\frak{D}),$ the relations~(\ref{eq2*!})--(\ref{eq8B_2}) hold at
$z_0.$ Therefore, by~(\ref{eq2A})
\begin{equation}\label{eq10_2_A}
M_p(f_m(\Gamma_m))\leqslant \int\limits_{A(z_0, r_m, r_1)} Q(x)\cdot
\eta^p(|x-z_0|)\ dm(x)
\end{equation}
for every Lebesgue measurable function $\eta:(r_m, r_1)\rightarrow
[0,\infty ]$ such that $\int\limits_{r_m}^{r_1}\eta(r)\,dr \geqslant
1.$
Observe that, the function
$$\eta(t)=\left\{
\begin{array}{rr}
\psi(t)/I(r_m, r_1), & t\in (r_m,
r_1),\\
0, & t\in {\Bbb R}\setminus (r_m, r_1)\,,
\end{array}
\right. $$ where $I(r_m, r_1):=\int\limits_{r_m}^{r_1}\psi(t)\,dt.$
Therefore, it follows from~(\ref{eq3.7.2_2}) and (\ref{eq10_2_A})
that
\begin{equation}\label{eq11_2_A}
M_p(f_m(\Gamma_m))\leqslant \alpha(r_m)\rightarrow 0
\end{equation}
as $m\rightarrow \infty,$ where $\alpha(\varepsilon)$ is some
nonnegative function tending to zero as $\varepsilon\rightarrow 0,$
which exists due to the condition~(\ref{eq3.7.2_2}).

\medskip
On the other hand, observe that $f_m(\Gamma_m)=\Gamma(C_m, f_m(A_m),
D_m^{\,\prime}).$ By the hypothesis of the lemma,
$h(f_m(A_m))\geqslant \delta$ for all $m\in {\Bbb N}.$ Therefore,
due to (\ref{eq5.1_A}), $h(f_m(A_m))\geqslant \delta_1$ and
$h(C_m)\geqslant\delta_1,$ where $\delta_1:=\min\{\delta, a/2\}.$
Using the fact that the domains $D_m^{\,\prime}:=f_m(D_m)$ are
equi-uniform with respect to the $p$-modulus, we conclude that there
exists $\sigma>0$ such that
$$M_p(f_m(\Gamma_m))=M_p(\Gamma(C_m, f_m(A_m),
D_m^{\,\prime}))\geqslant \sigma\qquad\forall\,\, m\in {\Bbb N}\,,$$
which contradicts the condition~(\ref{eq11_2_A}). The obtained
contradiction indicates that the assumption of the absence of
equicontinuity of the family $\frak{F}_{Q, A, p, \delta}(D_0,
\frak{D})$ at $z_0$ was incorrect. The obtained contradiction
completes the proof of the uniform equicontinuity of $f_m,$
$m=1,2,\ldots,$ in $\frak{D}$ in terms of prime ends.

\medskip
The statement of Lemma~\ref{lem2_A} concerning item~\textbf{I} is
established. The items~\textbf{II} and~\textbf{III} may be
considered similarly to how it was done under the proof of
Theorem~\ref{th2}.~$\Box$
\end{proof}

\medskip
Now, we give the version of Lemma~\ref{lem2_A} for mappings with a
branching.

\medskip
\begin{lemma}\label{lem7}
{\it\, The statement of Theorem~\ref{th5} remains true, if, under
assumptions of this theorem, instead of the assumption~4) mentioned
here we require the following: there exists
$\varepsilon_0=\varepsilon_0(x_0)>0$ and a Lebesgue measurable
function $\psi:(0, \varepsilon_0)\rightarrow [0,\infty]$ such
that~(\ref{eq7***_2}) holds while the relation~(\ref{eq3.7.2_2}) is
satisfied as $\varepsilon\rightarrow 0,$ where $A(x_0, \varepsilon,
\varepsilon_0)$ is defined~in (\ref{eq1**}).}
\end{lemma}

\medskip
\begin{proof} \textbf{I.} The proof is much the same as that
of Lemma~\ref{lem2_A}, so we will limit ourselves to a schematic
proof.

\medskip
Assume the contrary. Then there exists $a>0$ such that for
$\delta=1/k,$ $k=1,2,\ldots,$ there are $m_k\in {\Bbb N}$ and $x_k,
x^{\,\prime}_k\in D_{m_k}$ such that $\rho(x_k,
x^{\,\prime}_k)<1/k,$ however, $h(f_{m_k}(x_k),
f_{m_k}(x^{\,\prime}_k))\geqslant a.$ By the definition,
$\overline{D_0}_P$ is a compact space, so that we may consider that
the sequences $x_k, x^{\,\prime}_k\rightarrow P_0\in
\overline{D_0}_P$ by the metrics $\rho.$ Arguing similarly to the
proof of Lemma~\ref{lem2_A}, we may consider that~(\ref{eq6***_A})
holds. If $P_0\in D_0$, then $f_m$ is equicontinuous at $P_0$ (see
\cite[Lemma~4.9]{RS$_1$} for $p=n$ and \cite[Lemma~4.3]{Sev$_1$} for
$n-1<p<n.$ Now, we may consider that $P_0\in
E_{D_0}=\overline{D_0}_P\setminus D_0.$ Under notations of
Lemma~\ref{lem2_A} we may consider that $x_{m}, x^{\,\prime}_{m}\in
d^{\,\prime}_m\cap D_{m}$ while $d^{\,\prime}_m\cap D_{m}$ is
connected. We join the points $x_m$ and $x^{\,\prime}_m$ by a path
$\gamma_m:[0,1]\rightarrow D_m\cap d^{\,\prime}_m$ such that
$\gamma_m(0)=x_m,$ $\gamma_m(1)=x^{\,\prime}_m$ and $\gamma_m(t)\in
D_m\cap d^{\,\prime}_m$ for $t\in (0,1).$ Set
$C_m:=|f_m(\gamma_m)|.$ Now, $h(C_m)\geqslant a/2$ for sufficiently
large $m\in {\Bbb N}.$

\medskip
By the definition of $\frak{R}_{Q, \delta, p, E}(D_0, \frak{D}),$
for any $f_m\in\frak{R}_{Q, \delta, p, E}(D_0, \frak{D})$ and any
domain $D^{\,\prime}_m:=f_m(D_m)$ there exists a continuum
$K_m\subset D^{\,\prime}_m$ such that $h(K_m)\geqslant \delta$ and
$h(f_m^{\,-1}(K_m),
\partial D_m)\geqslant \delta>0.$
Since by the assumption the domains $D^{\,\prime}_m$ form the
equi-uniform family with respect to $p$-modulus and $h(C_m)\geqslant
a/2$ we obtain that
\begin{equation}\label{eq13_5}
M_p(\Gamma(K_m, C_m, D^{\,\prime}_m))\geqslant b\,,
\end{equation}
$m=1,2,\ldots ,$ for some $b>0.$
Consider the family $\Gamma_m,$ consisting of all paths $\beta:[0,
1)\rightarrow D^{\,\prime}_m,$ where $\beta(0)\in C_m$ and
$\beta(t)\rightarrow p\in K_m$ as $t\rightarrow 1.$ Let $\Gamma^*_m$
be a family of all total $f_m$-liftings $\alpha:[0, 1)\rightarrow D$
of the family $\Gamma_m$ with origin at $|\gamma_m|$ (see
Proposition~\ref{pr4_a}). Since $f_m$ is closed,
$\alpha(t)\rightarrow f^{\,-1}_m(K_m),$ where $f^{\,-1}_m(K_m)$ is
the total preimage of the continuum $K_m$ under the mapping $f_m.$
Since $h(f_m^{\,-1}(K_m),
\partial D_m)\geqslant \delta>0$ and $D_m\subset D_0,$ by
Proposition~\ref{pr6}, $h(f_m^{\,-1}(K_m),
\partial D_0)\geqslant \delta>0,$ as well. Now,
we may consider that $f_m^{\,-1}(K_m)\subset D_0\setminus
d^{\,\prime}_1.$ In this case, by Proposition~\ref{pr2}
\begin{equation}\label{eq2A_1}
\Gamma^*_m>\Gamma(S(z_0, r_m), S(z_0, r_1), D_0)\,.
\end{equation}
It follows by~(\ref{eq2A_1}) and~(\ref{eq2*!}) that,
\begin{gather}M_p(f_m(\Gamma_m^*))\leqslant M_p(\Gamma(S(z_0, r_m), S(z_0, r_1),
D_0))\nonumber \\
\label{eq10AB} \leqslant M_p(f_m(\Gamma(|\gamma_m|, f^{\,-1}_m(K_m),
D_m)))\leqslant \int\limits_{A(x_0, r_m, r_1)} Q(x)\cdot
\eta^p(|x-z_0|)\ dm(x)
\end{gather}
for each measurable function $\eta:(r_m, r_1)\rightarrow [0,\infty
],$ such that $\int\limits_{r_1}^{r_m}\eta(r)\,dr \geqslant 1.$
Again, arguing similarly to the proof of Lemma~\ref{lem2_A}, we
obtain from~(\ref{eq10AB}) that
\begin{equation}\label{eq11B_2}
M_p(f_m(\Gamma^{\,*}_m))\leqslant \alpha(r_m)\rightarrow 0
\end{equation}
as $m\rightarrow \infty,$ where $\alpha(\varepsilon)$ is some
nonnegative function tending to zero as $\varepsilon\rightarrow 0$
(which exists by the condition (\ref{eq3.7.2_2})). In addition,
$f_m(\Gamma^{\,*}_m)= \Gamma_m$ and $M_p(\Gamma_m)=M_p(\Gamma(K_m,
C_m, D^{\,\prime}_m)).$ Thus
\begin{equation}\label{eq12_5}
M_p(f_m(\Gamma^{\,*}_m))=M_p(\Gamma(K_m, C_m, D^{\,\prime}_m))\,.
\end{equation}
However, relations (\ref{eq11B_2}) and (\ref{eq12_5}) contradict
with~(\ref{eq13_5}). The resulting contradiction indicates that the
assumption (\ref{eq6***_A}) was incorrect, and, therefore, the
family $\frak{R}_{Q, \delta, p, E}(D_0, \frak{D})$ is equicontinuous
at $P_0\in E_{D_0}.$

\medskip
The statement of Lemma~\ref{lem7} concerning item~\textbf{I} is
established. The item~\textbf{II} may be considered similarly to how
it was done under the proof of Theorem~\ref{th4_4}.~$\Box$
\end{proof}

\medskip
{\it Proof of Theorem~\ref{th4}} follows by Lemma~\ref{lem2_A} and
Proposition~\ref{pr1}.~$\Box$

\medskip
{\it Proof of Theorem~\ref{th5}} follows by Lemma~\ref{lem7} and
Proposition~\ref{pr1}.~$\Box$

\medskip
\begin{example}\label{ex2}
Let $D_0$ be the unit disk ${\Bbb D}$ in ${\Bbb C}$ with a cut along
the segment $[0, 1]$, i.e., $D_0={\Bbb D}\setminus I \subset {\Bbb
C}$, where ${\Bbb D}:=\{z\in {\Bbb C}\,:\, |z|<1\}$, $I:=\{z=x+iy\in
{\Bbb C}\,:\, y=0, 0\leqslant  x\leqslant  1\}$. Let $f_m$ be a
conformal mapping of $D_m=\{z\in {\Bbb C}\,:\,
|z|<m/(m+1)\}\setminus I_m,$ $I_m:=\{z=x+iy\in {\Bbb C}\,:\, y=0,
0\leqslant  x\leqslant m/(m+1)\},$ onto ${\Bbb D}.$ For instance, we
may take
$$f_m(z)=\frac{\left(\frac{\sqrt{zr_m}+1}{\sqrt{zr_m}-1}\right)^2-i}
{\left(\frac{\sqrt{zr_m}+1}{\sqrt{zr_m}-1}\right)^2-i}\,,\quad
z=x+iy\,,\quad i^2=-1\,,$$
where $\sqrt{z}$ denotes the main branch of the square root and
$r_m=(m+1)/m.$ We verify conditions of Theorem~\ref{th4}. First of
all, $f_m$ are homeomorphisms of $D_m$ onto ${\Bbb D}$
satisfying~(\ref{eq2*!})--(\ref{eq8B_2}) at any point of
$\overline{D_0}=\overline{\Bbb D}$ with $Q\equiv 1$ and $p=n=2$ as
conformal mappings (see \cite[Theorem~8.1]{Va}). Since this sequence
(obviously) converges to the conformal mapping
$f(z)=\frac{\left(\frac{\sqrt{z}+1}{\sqrt{z}-1}\right)^2-i}
{\left(\frac{\sqrt{z}+1}{\sqrt{z}-1}\right)^2-i}$ locally uniformly
in $D_0,$ the relation $h(f_m(A))\geqslant\delta$ obviously holds
for any $m=1,2,\ldots ,$ any continuum $A$ in
$D_0\cap\bigcap\limits_{m=1}^{\infty}D_m$ and some $\delta>0.$
Besides that, the relation $h(\overline{{\Bbb R}^n}\setminus
f_m(D_m))=h(\overline{{\Bbb R}^n}\setminus {\Bbb D})\geqslant
\delta$ holds, for in0stance, with $\delta=1.$ Observe that $D_m,$
$m=1,2,\ldots ,$ converge to $D_0={\Bbb D}\setminus I$ as its
kernel, see Figure~\ref{fig2}.
\begin{figure}
  \centering\includegraphics[scale=0.5]{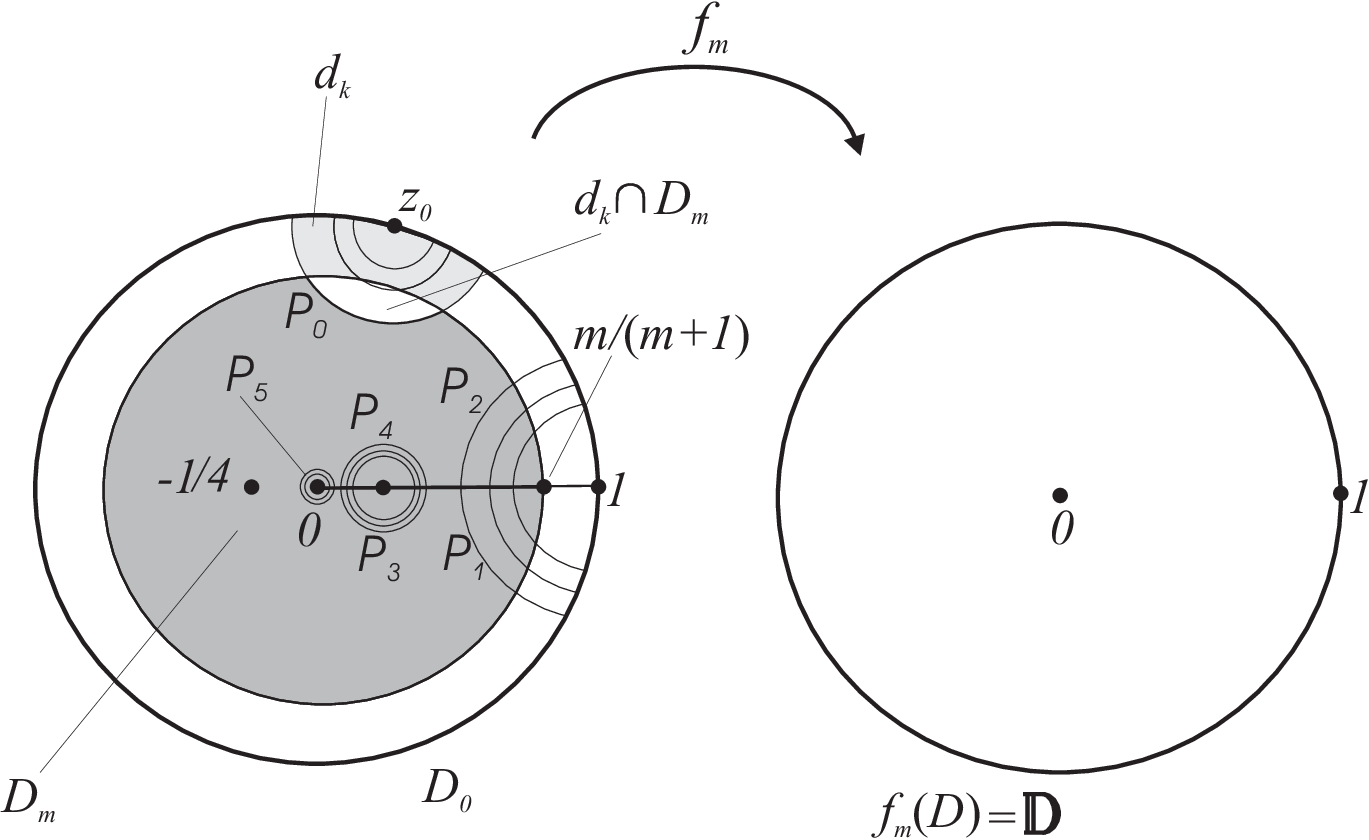}
  \caption{Illustration for Example~\ref{ex2}.}\label{fig2}
 \end{figure}
Besides that, the domains $d_k,$ formed by chains of cuts lying on
the circles centered at some point $z_0\in
\partial {\Bbb D}\cup I$ form a connected intersection with the
domains $D_m$ for each fixed $k$ and all sufficiently large $m.$
Indeed, prime ends of $D_0$ may be conditionally divided into~6
groups: 0) the set of prime ends $P_0$ which correspond to $z_0\in
{\Bbb S}^1\setminus I,$ where ${\Bbb S}^1=\partial {\Bbb D};$ 1) and
2): the prime ends $P_1$ and $P_2$ with impression $z_0=1,$ the
sequences of domains $d_k$ lie from below and from upper of $I,$
correspondingly; 3) and 4): the set of prime ends $P_3$ and $P_4$
with impression at some $z_0\in I\setminus \{1,0\},$ the sequences
of domains $d_k$ lie from below and from upper of $I,$
correspondingly; 5) the prime end $P_5$ which correspond to its
impression $I(P_5)=\{0\}.$ Let $d_k,$ $k=1,2,\ldots,$ be a sequence
of domains in $P_i,$ $i=\overline{0, 5}.$ Now, $d_k$ is convex for
any $k=1,2,\ldots. $ Given $k\in {\Bbb N},$ the intersection
$D_m\cap d_k$ is non-empty and convex for sufficiently large $m.$ In
particular, $D_m\cap d_k$ is connected for above $m,$ as required.
Thus, 1) the sequence of domains $D_m$ is regular with respect
to~$D_0$ in terms of prime ends. Observe that, 2) for every $m\in
{\Bbb N},$ a domain $D_m$ is quasiconformally regular; $D_0$ is also
quasiconformally regular. Indeed, $D_0$ and $D_m$ are regular
because by the Riemannian mapping theorem, $D_0$ (or $D_m$) is
conformally equivalent to ${\Bbb D}$; moreover, ${\Bbb D}$ has
locally quasiconformal boundary (see e.g. \cite[Theorem~17.10]{Va}).
Note that, 3) the family $f_m(D_m)={\Bbb D}$ is equi-uniform with
respect to $2$-modulus over all $m\in {\Bbb N}$ (see
Remark~\ref{rem1}). Finally, 4) $Q\equiv 1$ has a finite mean
oscillation in $\overline{D_0};$ moreover, (\ref{eq2_4}) holds
for every $x_0\in \overline{D_0}$ and any $\beta(x_0)>0.$ Thus, all
of conditions of Theorem~\ref{th4} hold, as required.
\end{example}

\medskip
\begin{example}\label{ex5}
It is possible to construct corresponding families of mappings,
satisfying Theorem~\ref{th4} that have unbounded characteristics.
Let $z_0\in D_0$ and $0<r_0<d(z_0, \partial D_0),$ where $D_0$ is a
domain from Example~\ref{ex2}. Put
$h(x)=\frac{x}{|x|\log\frac{(r_0e)}{|x|}},$ $x\in B(z_0, r_0),$
$h(z_0)=z_0,$ $h|_{S(z_0, r_0)}=x.$ Then $h$ is defined in the ball
$B(z_0, r_0)$ and $h(B(z_0, r_0))=B(z_0, r_0).$ Reasoning similarly
to \cite[Proposition~6.3]{MRSY}, it may be shown that $h$ satisfies
the relations~(\ref{eq2*!})--(\ref{eq8B_2}) at any point $x_0\in
\overline{B(z_0, r_0)}$ for
$Q=Q(x)=\log\left(\frac{r_0e}{|x|}\right).$ Note that $Q$ satisfies
(\ref{eq2_4}) for every $x_0\in \overline{D_0}$ and any
$\beta(x_0)>0$ that may be verified directly. Now, we put
$$g_m(x)=\begin{cases}(f_m \circ h)(x)\,,& x\in B(z_0, r_0)\,,\\
f_m(x)\,,& x\not\in B(z_0, r_0)\end{cases}\,,$$
%
%
where $f_m$ and $f$ are from Example~\ref{ex2}.
By the construction, $g_m$ satisfy the
relations~(\ref{eq2*!})--(\ref{eq8B_2}) at any point $x_0\in
\overline{D_0}$ for $Q= Q(y)=\log\left(\frac{r_0e}{|x|}\right).$
Observe that, the domains $f_m(D_m)$ are not changed, so that the
domains $g_m(D_m)$ are regular, as well. In addition, since $h$ is a
fixed mapping, the relations $h(g_m(A))\geqslant\delta,$
$m=1,2,\ldots ,$ also hold for some $\delta>0$ and all $m=1,2,\ldots
.$ The mappings $g_m,$ $m=1,2,\ldots ,$ are homeomorphisms and
satisfy all the conditions of Theorem~\ref{th4}.
\end{example}

\medskip\medskip\medskip
{\bf Declarations.}

\medskip
{\bf Conflict of interest.} The authors have no financial or
proprietary interests in any material discussed in this article.

\medskip
{\bf Funding.} The work was supported by the National Research
Foundation of Ukraine, Project number 2025.07/0014, Project name:
``Modern problems of Mathematical Analysis and Geometric Function
Theory''.

\medskip
{\bf Data availability statement.} All necessary data are included
in the paper.

\medskip
{\bf Contribution.} The authors have equal contribution to writing
the article.

\medskip
{\bf \noindent Nataliya Ilkevych} \\
Zhytomyr Ivan Franko State University,  \\
40 Velyka Berdychivs'ka Str., 10 008  Zhytomyr, UKRAINE \\
Email: ilkevych1980@gmail.com

\medskip
{\bf \noindent Denys Romash} \\
Zhytomyr Ivan Franko State University,  \\
40 Velyka Berdychivs'ka Str., 10 008  Zhytomyr, UKRAINE \\
dromash@num8erz.eu

\medskip
{\bf \noindent Evgeny Sevost'yanov} \\
{\bf 1.} Zhytomyr Ivan Franko State University,  \\
40 Velyka Berdychivs'ka Str., 10 008  Zhytomyr, UKRAINE \\
{\bf 2.} Institute of Applied Mathematics and Mechanics\\
of NAS of Ukraine, \\
19 Henerala Batyuka Str., 84 116 Slov'yansk,  UKRAINE\\
esevostyanov2009@gmail.com

\end{document}